%% file: realizability.tex
\documentclass[reqno]{amsart}

\usepackage{amsfonts}
\usepackage{amscd}
\usepackage{amsbsy}
\usepackage{amsxtra}
\usepackage{amssymb}
\usepackage{epsfig}
\usepackage{epic,eepic}
\usepackage{enumerate}
\usepackage{graphicx}
\usepackage[all]{xy}
\usepackage{xypic}
\usepackage{psfrag}
\usepackage{amsmath}
\usepackage{amsthm}
\usepackage{extarrows}
\usepackage{setspace}
\usepackage{hyperref}
\usepackage{url}
\usepackage{here}
\usepackage{todonotes}
\newlength{\halfbls}
\usepackage{xspace}
\usepackage{stmaryrd}
\usepackage{tikz}
\usepackage{MnSymbol}
\usetikzlibrary{calc}
\usetikzlibrary{decorations.pathreplacing,decorations.markings}
\usetikzlibrary{positioning,arrows}
\usetikzlibrary{cd}
\usetikzlibrary{intersections}
\DeclareRobustCommand{\SkipTocEntry}[9]{}

\include{macros_realize}

\usepackage[maxnames = 5, style=alphabetic,
            isbn=false,
            doi=false,
            url=false,
            backend=bibtex,
           ]{biblatex}
\defbibheading{bibliography}[\bibname]{%
  \addcontentsline{toc}{section}{References}%
  \addtocontents{toc}{\SkipTocEntry}%
  \section*{References}%
}

\bibliography{my}

\title{Realizability of tropical canonical divisors}

\begin{document}

\author{Martin M\"oller}
\address{
Institut f\"ur Mathematik, Goethe--Universit\"at Frankfurt,
Robert-Mayer-Str. 6--8,
60325 Frankfurt am Main, Germany
}
\email{moeller@math.uni-frankfurt.de}
  
\author{Martin Ulirsch}
\address{Department of Mathematics\\ 
University of Michigan\\
530 Church Street\\
Ann Arbor, MI 48109\\
USA}
\email{ulirsch@umich.edu}

\author{Annette Werner}
\address{
Institut f\"ur Mathematik, Goethe--Universit\"at Frankfurt,
Robert-Mayer-Str. 6--8,
60325 Frankfurt am Main, Germany
}
\email{werner@math.uni-frankfurt.de}
\begin{abstract}
We use recent results by Bainbridge-Chen-Gendron-Gru\-shevsky-M\"oller on compactifications of strata of abelian differentials to give a comprehensive solution to the realizability problem for effective tropical canonical divisors in equicharacteristic zero. Given a pair $(\Gamma, D)$ consisting of a stable tropical curve $\Gamma$ and a divisor $D$ in the canonical linear system on $\Gamma$, we give a purely combinatorial condition to decide whether there is a smooth curve $X$ over a non-Archimedean field whose stable reduction has $\Gamma$ as its dual tropical curve together with a effective canonical divisor $K_X$ that specializes to $D$. Along the way, we develop a moduli-theoretic framework to understand Baker's specialization of divisors from algebraic to tropical curves as a natural toroidal tropicalization map in the sense of Abramovich-Caporaso-Payne.
\end{abstract}

\maketitle
\tableofcontents
\noindent
\SaveTocDepth{1} 

\input{sec_intro}

\input{sec_modulidivisors}

\input{sec_modulitropdivisors}

\input{sec_tropdivisors}

\input{sec_tropHodge}

\input{sec_twisteddifferentials}

\input{sec_CharDim.tex}

\printbibliography


\end{document}

%% file: macros_realize.tex
\newtheorem{Defi}{Definition}[section]   \newtheorem{definition}[Defi]{Definition}
   \newtheorem{Prop}[Defi]{Proposition} \newtheorem{proposition}[Defi]{Proposition}
\newtheorem{Lemma}[Defi]{Lemma}     
\newtheorem{Thm}[Defi]{Theorem}  \newtheorem{theorem}[Defi]{Theorem}
  
   \newtheorem{mainthm}{Theorem} 
  \newtheorem{Ex}[Defi]{Example}

\theoremstyle{definition}
 \newtheorem{remark}[Defi]{Remark} 

\newcommand{\CC}{\mathbb{C}}

\newcommand{\PP}{\mathbb{P}} 
\newcommand{\QQ}{\mathbb{Q}} 
\newcommand{\RR}{\mathbb{R}} 
\newcommand{\R}{\mathbb{R}}

\newcommand{\ZZ}{\mathbb{Z}}
\def\Z{\ZZ} 
\newcommand{\N}{\mathbb{N}}

\newcommand{\frako}{\mathfrak{o}}

\newcommand{\frakS}{\mathfrak{S}}

\newcommand{\frakM}{\mathfrak{M}}

 \newcommand{\cXX}{{\mathcal X}}

 \newcommand{\cOO}{{\mathcal O}}
 \newcommand{\cRR}{{\mathcal R}}

\newcommand{\calA}{\mathcal{A}}
\newcommand{\calD}{\mathcal{D}}
\newcommand{\calO}{\mathcal{O}}
\newcommand{\calM}{\mathcal{M}}
\newcommand{\calMbar}{\overline{\mathcal{M}}}
\newcommand{\Mbar}{\overline{M}}

\newcommand{\calI}{\mathcal{I}}
\newcommand{\calX}{\mathcal{X}}
\newcommand{\calY}{\mathcal{Y}}
\newcommand{\calXbar}{\overline{\mathcal{X}}}
\newcommand{\calH}{\mathcal{H}}
\newcommand{\calR}{\mathcal{R}}

\newcommand{\bfp}{{\bf p}}

\newcommand{\bfs}{{\bf s}}

\newcommand{\bfx}{{\bf x}}

\newcommand{\Aut}{{\rm Aut}}

\newcommand{\Rat}{{\rm Rat}}
\newcommand{\Div}{{\rm Div}}
\newcommand{\PDiv}{{\rm PDiv}}

\newcommand{\Hom}{{\rm Hom}}

\newcommand{\ord}{{\rm ord}}

\newcommand{\an}{{\rm an}}
\newcommand{\val}{{\rm val}}
\newcommand{\trop}{{\rm trop}}

\renewcommand{\div}{{\rm div}}

\DeclareMathOperator{\Res}{Res}
\DeclareMathOperator{\Sym}{Sym}
\newcommand{\calDiv}{\mathcal{D}iv}
\newcommand{\calDivbar}{\overline{\mathcal{D}iv}}
\DeclareMathOperator{\Spec}{Spec}
\DeclareMathOperator{\characteristic}{char}
\newcommand{\Divbar}{\overline{Div}}
\newcommand{\PGL}{PGL} 
\DeclareMathOperator{\mdeg}{mdeg}

\newcommand{\twd}{twisted differential\xspace}

\def\={\;=\;}  \def\+{\,+\,}

\newcommand{\ol}{\overline}

\newcommand{\aG}{\Gamma}                  
\newcommand{\oG}{{\overline{\Gamma}}}
\newcommand{\eG}{{{\Gamma}^+}}            
\newcommand{\lG}{{\overline{\Gamma}}}     

\def\be{\begin{equation}}   \def\ee{\end{equation}}     \def\bes{\begin{equation*}}    \def\ees{\end{equation*}}
\def\ba{\be\begin{aligned}} \def\ea{\end{aligned}\ee}   \def\bas{\bes\begin{aligned}}  \def\eas{\end{aligned}\ees}

\newcommand{\proj}{{\mathbb P}}
\newcommand{\moduli}[1][g]{{\mathcal M}_{#1}}
\newcommand{\omoduli}[1][g]{{\Omega \mathcal M}_{#1}}

\newcounter{savedtocdepth}
\newcommand*{\SaveTocDepth}[1]{%
  \addtocontents{toc}{%
    \protect\setcounter{savedtocdepth}{\protect\value{tocdepth}}%
    \protect\setcounter{tocdepth}{#1}%
  }%
}

%% file: sec_intro.tex
\section*{Introduction}

The \emph{realizability problem} in tropical geometry is a metaproblem that underlies many of the successful applications of tropical geometry to other areas of mathematics. It asks whether for a given  synthetically defined tropical object, there exists an analogous algebraic geometric object whose tropicalization is precisely the given tropical object. 

The realizability problem for divisors is, in general, notoriously difficult, see e.g.
\cite[Section~10]{BakerJensen}. In this article we solve it for effective
canonical divisors using recent results on the compactification of strata of abelian differentials in~\cite{bcggm}. 

\subsection*{Realizability of tropical canonical divisors} 
Let $(\Gamma, D)$ be a tuple  consisting of an (abstract) stable tropical
curve~$\Gamma$ and a divisor $D$ in the canonical linear system~$|K_\Gamma|$
on~$\Gamma$.
Does there exist a smooth curve $X$ together with a stable degeneration $\calX$ 
as well as an effective canonical divisor $K_X$ on~$X$ such that the following two conditions hold: 
\begin{itemize}
\item the tropical curve given by the metrized weighted dual graph 
of the irreducible components in the special fiber of $\calX$ 
is~$\Gamma$; and
\item the specialization of~$K_X$, i.e.\ the multidegree of the
special fiber of the closure of $K_X$ in a suitably chosen semistable model of~$\calX$,
is equal to~$D$. 
\end{itemize}
If that is the case, we say the pair $(\Gamma, D)$
is \emph{realizable}. 
\par
Our main theorem gives an exhaustive answer to this question 
over an algebraically closed field of characteristic~$0$. To state it,
recall that an
element in the tropical canonical series differs from the distinguished
element~$K_\Gamma$ by the divisor of a piecewise affine function~$f$ on~$\Gamma$
with integral slopes. We declare the support of $\div(f)$ to be vertices
of $\Gamma$ and add to the graph~$\Gamma$ legs at the support of  $\div(f)$,
according to the local multiplicity. Now we simply use the value of such a
function~$f$ to define an order among the vertices of~$\Gamma$ (making it 
into a \emph{level graph}). Finally, we provide each half-edge of~$\Gamma$
with an enhancement consisting of the (outgoing) slope of~$f$. The
resulting object is called the \emph{enhanced level graph} $\Gamma^+(f)$
associated with~$f$. 
\par
Section~\ref{sec:twd} explains the algebro-geometric origin of this notion.
The correspondence between rational functions and decorations on~$\Gamma$
is explained in Section~\ref{sec:descR}. In particular, we introduce the notion
of an \emph{inconvenient vertex}~$v$. A vertex $v \in \Gamma^+(f)$
of genus $0$ is inconvenient if it has, roughly speaking, an edge with a 'large'
positive decoration. For example, trivalent vertices where two edges
have decoration less than~$-1$ are always inconvenient.
\par
\begin{mainthm}\label{thm_realizabilitylocus}
Given a tropical curve $\Gamma$ and an element $D = K_\Gamma + \div(f)$
in the tropical canonical linear series in $\Gamma$, the pair $(\Gamma, D)$ is realizable if and only if the following two
conditions hold: 
\begin{itemize}
\item[i)] For every inconvenient vertex~$v$ of $\Gamma^+(f)$
there is a simple cycle $\gamma \subset \Gamma$ based at~$v$ 
that does not pass through any node at a level smaller than~ $f(v)$.
\item[ii)] For every horizontal edge~$e$ there is a simple cycle 
$\gamma \subset \Gamma$ passing though~$e$ which
does not pass through any node at a level smaller than $f(e)$.
\end{itemize}
\end{mainthm}
\par
This theorem implies in particular that the canonical divisor~$K_\Gamma$
on~$\Gamma$  is in general not realizable (see Example~\ref{ex:dumbbell} below).
Note that $K_\Gamma$ is always the tropicalization of some (non-effective)
canonical divisor by \cite[Remark~4.21]{bakerSpec}. 

\subsection*{The realizability locus in the tropical Hodge bundle}
In \cite{LinU} Lin and the second author of this article synthetically constructed a tropical analogue $\PP\omoduli[g]^{trop}$ of the projective Hodge bundle. Set-theoretically it parametrizes isomorphism classes of pairs $(\Gamma, D)$ where $\Gamma$ is a stable tropical curve of genus $g$ and $D$ is an element of the canonical linear system on $\Gamma$. By \cite[Theorem~1.2]{LinU} it canonically carries the structure of a generalized (rational polyhedral) cone complex. We denote by $\PP\omoduli[g]^{an}$ the Berkovich analytic space
associated to $\PP\omoduli[g]$ in the sense of \cite{Ulirsch_tropisquot}.
 There is a natural tropicalization map 
\begin{equation*}
\trop_{\Omega}\,\colon\,\PP\omoduli[g]^{an}\longrightarrow \PP\omoduli[g]^{trop}
\end{equation*}
that sends an element in $\PP\omoduli[g]^{an}$, represented by a pair
$(X,K_X)$
consisting of a smooth algebraic curve $X$ over a non-Archimedean extension of
the base field and a canonical divisor $K_X$ on $X$, to the point
consisting of the dual tropical curve $\Gamma$ of a stable reduction~$\calX$
of~$X$ together with specialization of~$K_X$ to~$\Gamma$ (see
Section~\ref{section_tropHodge} below for details). 
\par
Theorem \ref{thm_realizabilitylocus} thus gives a complete characterization of the elements in the so-called \emph{realizability locus}, the image of $\trop_\Omega$ in $\PP\omoduli[g]^{trop}$.
\par
\medskip
In general, by the classical Bieri-Groves Theorem (see \cite[Theorem~A]{BieriGroves} and \cite[Theorem~2.2.3]{EinsiedlerKapranovLind}) the tropicalization of a
subvariety of a split algebraic torus is a rational polyhedral complex of the same
dimension. The Hodge bundle does not embed into an algebraic torus but rather in
a suitably defined toroidal compactification 
(in the sense of \cite{KKMSD_toroidal}).
Consequently, we know in this case a priori only that the dimension of the
realizability
locus is bounded above by $4g-4$ by \cite[Theorem~1.1]{Ulirsch_tropcomp}.
Our methods allow us to prove the following stronger result.  

\begin{mainthm}\label{thm_dimension}
The realizability locus in $\PP\omoduli[g]^{trop}$ naturally has the structure of a generalized rational polyhedral cone complex of pure dimension $4g-4$. 
\end{mainthm}

The main ingredient in the proofs of both Theorem \ref{thm_realizabilitylocus} and \ref{thm_dimension} is the  description of compactifications of strata of abelian differentials in \cite{bcggm}, which is achieved using the method of plumbing and glueing. So our proof only works in equicharacteristic zero. It would be highly interesting to find an purely algebraic-geometric proof of these results (and the ones in \cite{bcggm}) that generalizes to all characteristics. 
\par
\subsection*{Realizability locus for strata}
The Hodge bundle has a natural stratification by locally closed subsets 
\begin{equation*}
\PP\omoduli[g] \,= \, \bigcup_{m_1,\ldots, m_n}
\PP\omoduli[g](m_1,\ldots, m_n)
\end{equation*}
where the strata parametrize canonical divisors whose support has multiplicity profile $(m_1, \dots, m_n)$ for non-negative integers $m_i$ with $m_1+\cdots + m_n=2g-2$. In our proof we construct a realization of a tropical canonical divisor by an element in the open stratum $\PP\omoduli[g](1,\ldots, 1)^{an}$
of $\PP\omoduli[g]^{an}$. However, our criterion works exactly the
same way for the realizability 
by an element in a fixed stratum $\PP\omoduli[g](m_1,\ldots, m_n)^{an}$.
In Section~\ref{sec:Rstrata} below we discuss this criterion in detail and show, in an example, how it can be applied to study the realizability problem for Weierstrass points in genus $2$. 

\subsection*{Tropicalizing the moduli space of divisors} We denote by
$\calDiv_{g,d}$ the moduli space parametrizing pairs $(X,\widetilde{D})$
consisting of a smooth algebraic curve~$X$  of genus~$g$ and an effective
divisor~$\widetilde{D}$ on~$X$ of degree~$d$, i.e. the symmetric product of
the universal curve over $\calM_g$. In Section~\ref{section_modulitropdivisors},
we define a natural tropical analogue $\Div_{g,d}^{trop}$ which is a
set-theoretic moduli space of pairs $(\Gamma, D)$ consisting of a tropical
curve $\Gamma$ together with an effective divisor $D$ of degree~$d$
on~$\Gamma$. Let $\calDiv_{g,d}^{an}$ be the Berkovich analytic space
associated to $\calDiv_{g,d}$. There is a natural tropicalization map 
\begin{equation*}
\trop_{g,d}\,\colon\, \calDiv_{g,d}^{an}\longrightarrow \Div_{g,d}^{trop}
\end{equation*}
that takes a point in $\calDiv_{g,d}^{an}$ represented by a tuple
$(X,\widetilde{D})$, consisting of a curve~$X$ over a non-Archimedean
extension $K$ of $k$ and a divisor $\widetilde{D}$ of degree $d$, to the dual tropical
curve $\Gamma$ of a stable reduction $\calX$ of $X$ together with the
specialization of $\widetilde{D}$ to~$\Gamma$
(see Section~\ref{section_tropdivisor} for details). 
\par
In Section \ref{section_modulidivisors} below we give a reinterpretation of this procedure by constructing a modular toroidal compactification $\calDivbar_{g,d}$ of $\calDiv_{g,d}$, using a variant of Hassett's \cite{Hassett_weightedstablemoduli} moduli spaces of weighted stable curves and inspired by Caporaso's treatment of the universal Picard variety in \cite{Caporaso_universalPic}. By \cite{Thuillier_toroidal, acp},
associated with the toroidal coordinates at the boundary, there is a natural
strong deformation retraction $\bfp_{g,d}$ onto a \emph{non-Archimedean skeleton} $\frakS_{g,d}$ of $\calDiv_{g,d}^{an}$ that is functorial with respect to toroidal morphism and therefore, in particular, commutes with the forgetful map to $\calM_{g}$. Expanding on the main result of \cite{acp}, we prove in Section~\ref{section_tropdivisor} below the following theorem.
\par
\begin{mainthm}\label{mainthm_tropDiv}
The tropicalization map $\trop_{g,d}\colon \calDiv_{g,d}^{an}\rightarrow \Div_{g,d}^{trop}$ has a natural continuous section $J_{g,d}$ that induces an isomorphism between $\Div_{g,d}^{trop}$ and the non-Archimedean skeleton $\frakS_{g,d}$ of $\calDiv_{g,d}$ that makes the diagram 
\begin{center}
\begin{tikzcd}
\calDiv_{g,d}^{an} \arrow[rd, "\bfp_{g,d}"'] \arrow[rr, "\trop_{g,d}"] &  & \Div_{g,d}^{trop} \arrow[ld,"J_{g,d}", "\cong"'] \\
& \frakS_{g,d} &
\end{tikzcd}
\end{center}
commute. 
\end{mainthm}

Theorem \ref{mainthm_tropDiv}, in particular, implies that $\trop_{g,d}$, as defined above, is well-defined and continuous. Moreover, it gives a moduli-theoretic perspective on Baker's specialization map from \cite[Section~2C]{bakerSpec} in the spirit of \cite{acp}. It has often been suggested that the semicontinuity of the Baker-Norine rank for divisors on curves \cite[Corollary~2.11]{bakerSpec} should be connected to the continuity (or flatness) of the process of tropicalization. Theorem~\ref{mainthm_tropDiv} provides one place where we can make the notion of continuity precise. It would be highly interesting to further investigate this connection in more detail.

A moduli-theoretic approach, very similar to the one above, has implicitly appeared in the proof of \cite[Theorem 4.6]{JensenRanganathan_BrillNoetherfixedgonality}, where the authors work with the stacky symmetric product of the universal curve over $\calMbar_g$ (see Remark \ref{remark_JensenRanganathan} for details).

\subsection*{Instances of the realizability problem} 
The realizability problem for divisors (or divisor classes) of a certain fixed rank on tropical curves is central to the many applications of the tropical approach to limit linear series and has recently received a significant amount of attention (see \cite[Section 10]{BakerJensen} and the references therein). It is a crucial element 
in the tropical approach to the maximal rank conjecture for quadrics \cite{JensenPayne_tropicalindependenceII}, which is based on a realizability result coming from \cite{CartwrightJensenPayne_lifting}, as well as in the recents works on the Brill-Noether varieties for curves of fixed gonality \cite{Pflueger_BrillNoetherkgonalcurves,JensenRanganathan_BrillNoetherfixedgonality}. We also highlight \cite{Cartwright_Mnevuniversality}, in which the author shows that this realizability problem fulfills a version of Murphy's Law in the sense of Mn\"ev universality, and \cite{He_smoothinglimitlinearseries}, which connects the classical smoothing problem for limit linear series with the divisor theory on metrized curve complexes of Amini and Baker in \cite{AminiBaker_metrizedcurvecomplexes}.
 
\subsection*{Acknowledgements}  M.M.\ and A.W.\ acknowledge 
support from the LOEWE-Schwerpunkt ``Uniformisierte Strukturen
in Arithmetik und Geometrie''. M.U.\ would like to thank the Max Planck Institute for Mathematics in the Sciences in Leipzig and, in particular, the Non-Linear Algebra Group (with its director Bernd Sturmfels) for hosting him, while the majority of work on this project was undertaken. The authors thank Matt Bainbridge, Matt Baker, Dawei Chen,
Renzo Cavalieri, Quentin Gendron, Sam Grushevsky, Bo Lin, Diane Maclagan, and Dhruv Ranganathan for useful comments and suggestions. We particularly thank Dave Jensen for pointing out missing cases in an earlier version of Example \ref{example_K4}.

%% file: sec_modulidivisors.tex
\section{Compactifying the moduli space of algebraic divisors}\label{section_modulidivisors}

Fix $k$ to be an algebraically closed field. Let $X$ be a scheme over a $k$. Recall that a Cartier divisor $D$ on $X$ is \emph{effective} if it admits a representation $(U_i, f_i)$ such that $f_i\in\Gamma(U_i, \calO_X)$ and $f_i$ is a non-zero divisor. We may think of $D$ as the closed subscheme of $X$ whose ideal sheaf $I(D)=\calO(-D)$ is invertible and  generated by $f_i$ on $U_i$. 

Given a morphism $X\rightarrow S$ of schemes over $k$, we say that an effective divisor $D$ on $X$ is a \emph{relative effective Cartier divisor} if it is flat over $S$ when regarded as a subscheme of $X$. If $S$ is connected, the function $s\mapsto \deg D_s$ is constant and will be referred to as the \emph{degree of $D$}.

Inspired by both \cite{Caporaso_universalPic} and \cite{Hassett_weightedstablemoduli}, we define the following moduli space.

\begin{definition}
Let $g\geq 2$ and $d\geq 0$. Define $\calDivbar_{g,d}$ to be the fibered category over $\calMbar_g$ whose fiber over a family $\pi\colon X\rightarrow S$ of stable curves is the set of pairs $(X', D)$ consisting of a semistable model $X'$ of $X$, i.e. a semistable curve over $S$ with stabilization $X$,
 and a relative effective Cartier divisor $D$ on $X'$ such that 
\begin{enumerate}[(i)]
\item the support of $D$ does not meet the nodes in each fiber $X'_s$ of $\pi'\colon X'\rightarrow S$; and 
\item the twisted canonical divisor $K_{X'}+D$ is relatively ample.
\end{enumerate}
\end{definition}

Denote by $\calDiv_{g,d}$ the preimage of the locus $\calM_g$ of smooth curves. The goal of this section is to prove  the following Theorem \ref{thm_Divbar}.

\begin{theorem}\label{thm_Divbar}
The fibered category $\calDivbar_{g,d}$ is a Deligne-Mumford stack of dimension $N=3g-3+d$ that is smooth and proper over $k$. Its coarse moduli space $\Divbar_{g,d}$ is projective. The complement of $\calDiv_{g,d}$ in $\calDivbar_{g,d}$ is a divisor with (stack-theoretically) normal crossings and the forgetful morphism $\calDivbar_{g,d}\rightarrow \calMbar_g$
is toroidal.
\end{theorem}

Given a smooth curve $X$ over the field $k$, it is well-known that the restriction of $\calDiv_{g,d}$ to the point $[X]$ in $\calM_g$ is representable by the $d$-th symmetric product $\Sym^d X$ of $X$ (see e.g.\ \cite[Theorem 3.1.3]{Milne_Jacobians}). In this section we generalize this result to all of $\calDivbar_{g,d}$ using a variant of Hassett's moduli spaces of weighted stable curves \cite{Hassett_weightedstablemoduli} that automatically provides us with a compactification of $\calDiv_{g,d}$ with favorable properties. Over $\calM_g$ this specializes to an isomorphism between $\calDiv_{g,d}$ and the relative symmetric product 
\begin{equation*}
\Sym^d \calX_g=\calX_g\times_{\calM_g} \cdots \times_{\calM_g}\calX_g/S_d
\end{equation*}
of the universal curve $\calX_g$ over $\calM_g$. 

\begin{remark}\label{remark_JensenRanganathan}
In the proof of \cite[Theorem 4.6]{JensenRanganathan_BrillNoetherfixedgonality} Jensen and Ranganathan work with the symmetric product $\Sym^d\calXbar_g$ of the compactified universal curve $\calXbar_g$ over $\calMbar_g$ as a compactification of $\calDiv_{g,d}$. This compactification is not smooth and not even toroidal, in general; the authors get around this restriction by working with the stacky symmetric product 
\begin{equation*}
\big[\Sym^d\calXbar_g\big] \= \big[\calX_g^d/S_d\big]
\end{equation*}
whose boundary admits (stack-theoretically) toroidal singularities. Our compactification is a relative coarse moduli space of a toroidal resolution of the singularities of $\big[\Sym^d\calXbar_g\big]$.
\end{remark}

\subsection{Moduli of stable curves with unordered points of small weight}

Let $\calA=(a_1,\ldots, a_d)\in (0,1]\cap \QQ^d$ and $g\geq 0$ such that 
\begin{equation*}
2g-2+a_1+\ldots + a_d>0 \ . 
\end{equation*}
In \cite{Hassett_weightedstablemoduli} Hassett studies the moduli stack $\calMbar_{g,\calA}$ of weighted stable $d$-marked curves of genus $g$. The fiber of the stack $\calMbar_{g,\calA}$ over a scheme $S$ is the groupoid of proper flat morphisms $\pi\colon X\rightarrow S$ such that each geometric fiber is a nodal curve of genus $g$ together with $n$ (marked) sections $s_1,\ldots, s_d$ that fulfill the following two axioms:
\begin{itemize}
\item the sections $s_1,\ldots, s_d$ lie in the smooth locus of $\pi$ and for all $s_{i_1}, \ldots, s_{i_k}$ with non-empty intersections we have $a_{i_1}+\ldots+a_{i_k}\leq 1$; and
\item the $\QQ$-divisor $K_X + a_1 s_1 +\ldots + a_ns_n$ is relatively ample.
\end{itemize}

If $\calA=(1,\ldots,1)$ this is the classical Deligne-Knudsen-Mumford moduli stack $\calMbar_{g,n}$ of $n$-marked stable curves of genus $g$ (see \cite{delignemumford,Knudsen_projectivityII}). By \cite[Theorem~2.1]{Hassett_weightedstablemoduli} the stack $\calMbar_{g,\calA}$ is an irreducible Deligne-Mumford stack that is smooth and proper over $K$ and its coarse moduli space $\Mbar_{g,\calA}$ is projective. Moreover, it is well-known that the complement of the locus of smooth curves $\calM_{g,\calA}$ in $\calMbar_{g,\calA}$ has (stack-theoretically) normal crossings (see e.g. \cite[Theorem~1.1]{Ulirsch_tropHassett}). 

Let $g\geq 2$, $d\geq 0$, and $\calA=(\epsilon,\ldots, \epsilon)=\epsilon^d$ for $\epsilon >0$ such that $d\cdot \epsilon\leq 1$, e.g. $\epsilon=\frac{1}{d}$.
In other words, the sections $s_1,\ldots,s_d$ of an object in $\calMbar_{g,\epsilon^d}$ need not be disjoint, and we require each rational component to have at least two nodes and to contain a marked point if is has precisely two nodes (so that the underlying nodal curve is semistable). 

There is a natural functor $\calMbar_{g,\epsilon^d}\rightarrow \calMbar_g$ that forgets about all sections and contracts all rational components that only have two nodes. Moreover, there is a natural operation of $S_d$ on $\calMbar_{g,\epsilon^d}$ that permutes the sections and the forgetful map $\calMbar_{g,\epsilon^d}\rightarrow \calMbar_g$ factors through the stack quotient $\big[\calMbar_{g,\epsilon^d}/S_d\big]\rightarrow \calMbar_g$. 

\subsection{Relative coarse moduli spaces}

Let  $f\colon\calX\rightarrow\calY$ be a morphism of algebraic stacks that is locally of finite presentation and whose relative inertia stack $\calI(\calX/\calY)=\calX\times_{\calX\times_\calY\calX}\calX$ is finite. By \cite[Theorem~3.1]{AbramovichOlssonVistoli_twistedstablemaps} there is an algebraic stack $X$ together with a factorization
\begin{equation*}
f = \overline{f} \circ \pi\,:\, 
\calX\xlongrightarrow{\pi}X\xlongrightarrow{\overline{f}}\calY
\end{equation*}
with $\overline{f}$ representable that is initial among such factorizations. The algebraic stack~$X$ is called the \emph{relative coarse moduli space} of $\calX$ over $\calY$. We refer the reader to \cite[Section~3]{AbramovichOlssonVistoli_twistedstablemaps} for the construction and basic properties of relative coarse moduli spaces.

\begin{theorem}\label{thm_Divgd=Hassett}
The fibered category $\calDivbar_{g,d}$ is equivalent to the relative coarse moduli space of the stack quotient $\big[\calMbar_{g,\epsilon^d}/S_d\big]$ over $\calMbar_g$.
\end{theorem}

In other words, the natural map
\begin{equation*}
\big[\calMbar_{g,\epsilon^d}/S_d\big]\longrightarrow \calDivbar_{g,d}
\end{equation*}
makes $\calDivbar_{g,d}$ into the initial object among all factorizations of $\big[\calMbar_{g,\epsilon^d}/S_n\big]\rightarrow \calX\rightarrow \calMbar_g$ such that $\calX\rightarrow \calMbar_g$ is representable. Theorem \ref{thm_Divgd=Hassett} is a global version of the fact that the $d$-fold symmetric product of a smooth curve $X$ represents the space of effective divisors of degree $d$ on $X$. Our proof is an adaption of \cite[Theorem~3.13]{Milne_Jacobians} to this situation. 

\begin{proof}[Proof of Theorem \ref{thm_Divgd=Hassett}]
An object in $\calMbar_{g,\epsilon^d}$ is a family $\pi\colon X'\rightarrow S$ of semistable curves of genus $g$ with $d$ sections $s_1,\ldots, s_d$ (not necessarily disjoint) such that the twisted canonical divisor $K_{X'}+\epsilon s_1+\ldots+\epsilon s_d$ is relatively ample. Then the sum $D=\sum_{i=1}^d s_i$ defines a relative effective divisor of degree $d$ on $X$ whose support does not meet the nodes in each fiber $X'_s$. We call such a relative effective Cartier divisor \emph{split}. The pair $(X',D)$ is stable in the sense that the twisted canonical divisor $K_{X'}+D$ is ample. This defines a functor $\calMbar_{g,\epsilon^d}\rightarrow \calDivbar_{g,d}$ that is invariant under the natural operation of $S_d$ on $\calMbar_{g,\epsilon^d}$ and so we have a natural induced morphism $\big[\calMbar_{g,\epsilon^d}/S_d\big]\rightarrow \calDivbar_{g,d}$ of categories fibered in groupoids over $\calMbar_g$.

Conversely, let $D$ be a relative effective Cartier divisor on a family $\pi\colon X'\rightarrow S$ of semistable curves. Assume first that $D$ is split. Then we may write $D=\sum_{i=1}^d s_i$ for sections $s_1,\ldots, s_d$ (not necessarily disjoint) that do not meet the nodes in each fiber and for which the twisted canonical divisor $K_{X'}+D=K_{X'}+ s_1+\ldots + s_d$ is relatively ample. This implies that the divisor $K_{X'}+\epsilon s_1+\ldots+\epsilon s_d$ is also relatively ample and therefore the tuple $(\pi,s_1,\ldots, s_d)$ defines an object in $\calMbar_{g,\epsilon^d}$. 

In general, when $D$ is not split, we may perform a finite and flat base change $T\rightarrow S$ so that the pullback $D_T$ of $D$ to $\pi'\colon X'_T\rightarrow T$ is split. In this case, both pullbacks of $D'$ to the family $ X'_{T'}\rightarrow T'=T\times_ST$ are split and naturally isomorphic. This induces a (representable) morphism $T\rightarrow\big[\calMbar_{g,\epsilon^d}/S_d\big]$, both of whose restrictions to~$T'$ are naturally equivalent. 

By finite flat descent \cite[Section I, Theorem 2.17]{Milne_etalecohomology}, every faithfully flat morphism of finite type is a strict epimorphism. This means that for every algebraic stack $Z$ with a representable morphism to $\calMbar_g$ the sequence
\begin{equation*}
\Hom_{\calMbar_g}(S,Z)\longrightarrow \Hom_{\calMbar_g}(T,Z)\rightrightarrows\Hom_{\calMbar_g}(T',Z)
\end{equation*}
is exact. If we take~$Z$ to be the relative coarse moduli space of
$\big[\calM_{g,\epsilon^d}/S_d\big]$  we find an object in the relative coarse
moduli space of $\big[\calM_{g,\epsilon^d}/S_d\big]$ that maps to $(X', D)$.
This shows that  $\calDivbar_{g,d}$ is equivalent to the relative coarse moduli
space of $\big[\calM_{g,\epsilon^d}/S_d\big]$. 
\end{proof}

\subsection{Proof of Theorem \ref{thm_Divbar}}

Theorem \ref{thm_Divgd=Hassett} together with \cite[Theorem 3.1]{AbramovichOlssonVistoli_twistedstablemaps} implies that $\calDivbar_{g,d}$ is an irreducible algebraic stack that is proper over $k$.
The universal properties immediately show that the coarse moduli space $\Divbar_{g,d}$ is isomorphic to the geometric invariant theory quotient $\Mbar_{g,\epsilon^d}/S_d$ (i.e. to the coarse moduli space of the stack quotient $\big[\calMbar_{g,\epsilon^d}/S_d\big]$) and therefore it is projective, since $\Mbar_{g,\epsilon^d}$ is.  

We choose an "exhausting" family for $\calMbar_{g,\epsilon^d}$ as in \cite[Section 3.4]{Hassett_weightedstablemoduli}: Set 
\begin{equation*}
L\=\omega_\pi^\nu(\nu\epsilon s_1+\cdots +\nu\epsilon s_d)
\end{equation*}
where $\nu >0$ is the smallest integer such that $\nu\epsilon$ is integral, i.e.
$\nu=d$ if we choose $\epsilon=\frac{1}{d}$. Set $q =\deg(L^{\otimes 3}) \= 3(2g-2+d\epsilon)$ as well as $r=3g-3$. 
Denote by $H_0$ the Hilbert scheme of tuples $(s_1,\ldots s_d)$ in $\PP^r$ and by $H_1$ the Hilbert scheme parametrizing curves~$X$ of genus~$g$ and degree $q$ 
 in~$\PP^r$. 
 Consider the locally closed locus $U\subseteq H_1\times H_0$ 
 of points $\big([X],s_1,\ldots, s_d\big)$ such that 
\begin{itemize}
\item the curve $X$ is reduced and nodal;
\item the sections $s_1,\ldots, s_d$ are contained in the smooth locus of $C$; and
\item $\calO_X(1)=L^{\otimes 3}$. 
\end{itemize}
Notice that there is a natural surjective morphism $U\rightarrow \calMbar_{g,\epsilon^d}$
that is given by forgetting the embedding into $\PP^r$. 

The symmetric group $S_d$ acts on $H_0$ by permuting the entries. The induced
action on $H_1\times H_0$ stabilizes $U$ and we denote the quotient $U/S_d$ 
by~$\widetilde{U}$. The fiber $\widetilde{U}_{[X]}$ over a point~$[X]$ in~$H_1$ is precisely the open subscheme of the Hilbert scheme of points on $X$ parametrizing 
closed subschemes $D$ of $X$ of length~$d$  whose support 
does not meet the nodes of $X$, or equivalently the open subset of the symmetric 
product $\Sym^d X$ whose points do not meet the nodes of~$X$. 

Consider the product $H_1\times \widetilde{H}_0$ where $\widetilde{H}_0$ is the Hilbert scheme of zero-dimensional subscheme of $\PP^r$ of length $d$. The quotient $\widetilde{U}$ is precisely the locus of tuples $(X,D)$ in $H_1\times\widetilde{H}_0$ such that 
\begin{itemize}
\item the curve $X$ is reduced and nodal;
\item the support of $D$ is contained in the smooth locus of $X$; and
\item $\calO_X(1)=L^{\otimes 3}$, where $L=\omega_X(D)$. 
\end{itemize}
The natural map $\widetilde{U}\rightarrow \calDivbar_{g,d}$ is representable and surjective by \cite[Proposition~3.3]{Hassett_weightedstablemoduli}. It is smooth, since every automorphism of $(X,D)$ is induced from a projective automorphisms of $\PP^r$. In fact, we may realize $\calDivbar_{g,d}$ naturally a quotient of $H_1\times \widetilde{H}_0$ by  the natural operation of $\PGL_r$. Since at all geometric points this operation has reduced and finite stabilizers, the diagonal morphism of $\calDivbar_{g,d}$ is unramified and therefore it is a Deligne-Mumford stack. 
\par
Since the symmetric product of a nodal curve is smooth away from the nodes (see e.g. the argument in  \cite[Proposition 3.2]{Milne_Jacobians}) each $\widetilde{U}_{[X]}$ is smooth. This implies that $\widetilde{U}$ is smooth and so is $\calDivbar_{g,d}$. 

Let $\frako_k$ be the field $k$ if $\characteristic k=0$ and otherwise the unique complete regular local with residue field $k$ and maximal ideal generated by $p$ where $\characteristic k=p>0$ (using Cohen's structure theorem). A standard deformation-theoretic argument (in the framework of \cite[Section 3.3]{Hassett_weightedstablemoduli}) shows that the complete local ring of $\calDivbar_{g,d}$ at a closed point $[X',D]$ (with nodes $x_1,\ldots, x_k$) is given by
\begin{equation*}
\widehat{\calO}_{\calDivbar_{g,d}, [X',D]}
\=\frako_k\llbracket t_1,\ldots, t_N\rrbracket
\end{equation*}
where $N=3g-3+n$ and the locus where $x_i$ stays nodal is given by $t_i=0$ for $1\leq i\leq k$. This shows that the complement of the smooth locus $\calDiv_{g,d}$ in $\calDivbar_{g,d}$ has (stack-theoretically) normal crossings.

Denote by $[X]$ in $\calMbar_g$ the stabilization of the curve $[X']$. The stability condition in place can be interpreted as saying that $X'$ has no rational tails and that the nodes of $X'$ have to contract to nodes of $X$. By \cite{delignemumford} the complete local ring of $\calMbar_g$ at the point $[X]$ is given by 
\begin{equation*}
\widehat{\calO}_{\calMbar_{g}, [X]} \=\frako_k\llbracket\widetilde{t}_1,\ldots, \widetilde{t}_{3g-3}\rrbracket
\end{equation*}
where the locus where the singularity $\widetilde{x}_i$ of stays nodal is given by $\widetilde{t}_i=0$ for $1\leq i\leq r$. Suppose now that the nodes $x_{i_1},\ldots, x_{i_l}$ on $X'$ are being contracted to $\widetilde{x}_i$ in $X$. Then the forgetful map $\calDivbar_{g,d}\rightarrow \calMbar_g$ is given by $\widetilde{t}_i=t_{i_1}\cdots t_{i_l}$ on the complete local rings and so it is a
toroidal morphism.

%% file: sec_modulitropdivisors.tex
\section{Tropical divisors and their moduli}\label{section_modulitropdivisors}

Let us first introduce tropical curves (see e.g. \cite{MikhICM}). A \emph{metric graph} is an equivalence class of tuples $(G,\vert\cdot\vert)$
consisting of a connected finite graph $G=(V,E)$ together with an edge length function
$\vert\cdot\vert: E(G)\rightarrow \RR_{>0}$. Two such tuples $(G,\vert\cdot\vert)$ and $(G',\vert\cdot\vert')$ are \emph{equivalent}, if there is a common length preserving refinement. We implicitly identify a metric graph, represented by $(G,\vert\cdot\vert)$, with its realization as a metric space, by glueing an interval of length $\vert e\vert$ for every edge $e$ according to the incidences in $G$. 

A \emph{tropical curve} $\Gamma$ is a metric graph together with a  function $h: \Gamma\rightarrow \Z_{\geq 0}$ with finite support. We refer to a tuple
$(G,\vert\cdot \vert)$ as a \emph{model} of $\Gamma$ if it represents $\Gamma$ as a metric graph, and if $h$ is supported on the vertices of $G$. The \emph{genus} of a tropical curve is defined to be 
\begin{equation} \label{eq:genus}
g(\Gamma) \= b_1(\Gamma)+\sum_{p\in \Gamma}h(p) \ .
\end{equation}
A model $G$ of a tropical curve is said to be \emph{semistable}, if for every vertex $v$ of $G$ we have $2h(v)-2+\vert v\vert \geq 0$, where $\vert v\vert$ denotes the valency of the vertex $v$. It is called \emph{stable}, if the above inequality is strict, i.e. if we have $2h(v)-2+\vert v\vert > 0$ for all vertices $v$ of $G$. Notice that, when a tropical curve $\Gamma$ admits a semistable model, its minimal model is necessarily stable. In this case, we call $\Gamma$ \emph{stable}. 

Later we will also use the notion of a \emph{tropical curve~$\Gamma$ with legs} for a tropical curve $\Gamma$ decorated with a collection~$L$ of infinite half-edges, i.e. of legs, emanating from the vertices of $G$. In this case, we modify the definition of \emph{stability}  by also counting the legs, when determining the valency of a vertex. Whenever it is clear from the context, we  refer to a tropical curve with legs simply as a \emph{tropical curve}. 

Let $g\geq 2$. The moduli space $M_g^{trop}$ of stable tropical curves is defined to be the set of isomorphism classes of stable tropical curves (without legs) of genus $g$. By \cite{acp} it has the structure of a \emph{generalized (rational polyhedral) cone complex}, i.e. it arises as a colimit of a diagram of (not necessarily proper) face morphisms of rational polyhedral cones. 

The goal of this section is to construct a moduli space $\Div_{g,d}^{trop}$ 
of tropical divisors of degree~$d$ over $M_g^{trop}$. 

\subsection{Divisors on tropical curves}
A {\em divisor} on a tropical curve~$\Gamma$ is a finite formal sum $D = \sum a_i p_i$ of points $p_i \in \Gamma$ with integral coefficients~$a_i$. We let $\deg(D) = \sum a_i$ be the degree of the divisor and write $D(p)=\sum_{p_i=p}a_i$ for a point $p\in\Gamma$. A divisor~$D$ is said to be {\em effective}, denoted by $D \geq 0$, if $D(p)\geq 0$ for all points $p$ of $\Gamma$. Given a tropical curve $\Gamma$, we denote by $\Div(\Gamma)$ the group of divisors on $\Gamma$. 

A {\em rational function} on $\Gamma$ is a continuous function $f : \Gamma \to \RR$ whose restriction to every edge $\Gamma$ (thought of as an interval $[0,\vert e\vert]$) is a piece-wise linear function whose slopes are integral.
Given a rational function~$f$ on $\Gamma$ and $P \in \Gamma$, we define the order $\ord_P(f)$ of~$f$ at~$P$ to be the sum of the outgoing slopes of~$f$ over all edges emanating from~$p$. This defines a map 
\begin{equation*}\begin{split}
{\rm div}: \Rat(\Gamma) &\longrightarrow \Div(\Gamma)\\
f&\longmapsto \sum_{p\in\Gamma}\ord_p(f) \cdot p
\end{split}\end{equation*} that assigns to any rational function its divisor. The image of the map ${\rm div}$ is called the subgroup $\PDiv(\Gamma) \subset \Div(\Gamma)$ of {\em principal divisors}. The divisors~$D$ and~$D'$ are called {\em equivalent} (denoted by $D \sim D'$) if $D-D' \in \PDiv(\Gamma)$.

We can now define the linear system of a divisor $D$ to be
\begin{equation*}
|D| \= \{D' \in \Div(\Gamma): D\geq 0 \textrm{ and }D \sim D' \} \ .
\end{equation*}
It is convenient to also introduce the tropical analogue
\begin{equation*}
 R(D) \= \{ f \in \Rat(\Gamma): D + {\rm div}(f) \geq 0\}
\end{equation*}
of the global sections of $\cOO(D)$. Note that we can shift any element in $R(D)$ by adding a real number and that $|D| = R(D)/\RR$.
\par
For any divisor~$D$ the space $R(D)$ has the structure of a polyhedral complex
(see e.g.\ \cite[Lemma 1.9]{GaKe}, \cite{MiZh}, and \cite[Proposition~3.2]{LinU}). However,
this polyhedral complex is not equidimensional in general, as we will see
in the case of the canonical linear system in Section~\ref{sec:descR}.

\subsection{Moduli of tropical divisors}

\begin{definition}
Let $g\geq 2$. The moduli space $\Div_{g,d}^{trop}$ is the set of isomorphism classes of tuples $(\Gamma, D)$ consisting of a stable tropical curve of genus $g$ and an effective divisor $D$ on $\Gamma$ of degree $d$. 
\end{definition}

\begin{proposition}\label{prop_Divgd=generalizedconecomplex}
The moduli space $\Div_{g,d}^{trop}$ naturally has the structure of a generalized cone complex of dimension $3g-3+d$.
\end{proposition}

Consider a pair $(G',D)$ consisting of a finite semistable vertex-weighted graph $G'$ of genus $g$ 
 and an effective divisor $D$ on $G'$ of degree $d$ supported on the vertices of $G'$. We say that the pair $(G',D)$ is \emph{stable} if for every vertex $v$ of $G'$ we have $2h(v)-2+\vert v\vert +D(v)>0$. 

\begin{proof}[Proof of Proposition \ref{prop_Divgd=generalizedconecomplex}]
Denote by $J_{g,d}$ the category of stable pairs $(G',D)$ where $G'$ is of genus $g$ and $D$ has degree $d$. 
The morphisms in $J_{g,d}$ are generated by 
\begin{itemize}
\item automorphisms $\phi$ of the weighted graph $G'$ such that $\phi^\ast D=D$; and
\item weighted edge contractions $\pi: (G_1',D_1)\rightarrow (G_2',D_2)$ (i.e. edge contractions for which $g\big(\pi^{-1}(v)\big)=h(v)$ for all vertices $v$ of $G_2'$) that fulfill $\pi_\ast D_1=D_2$. 
\end{itemize}

There is a natural functor from $J_{g,d}$ to the category $\mathbf{RPC}^{face}$ of (rational polyhedral) cones with (not-necessarily proper) face morphisms, given by $(G',D)\mapsto \sigma_{G'}=\R_{\geq 0}^{E(G')}$. Recall that a \emph{face morphism} is a morphism of rational polyhedral cones $\sigma\rightarrow \sigma'$ that induces an isomorphism between $\sigma$ and a (not necessarily proper) face of $\sigma'$; the class of face morphisms, in particular, includes all automorphisms. 

For a fixed $(G',D)$, the open cone $\mathring{\sigma}_{G'}=\R_{>0}^{E(G')}$ parametrizes the space of triples consisting of
\begin{itemize}
\item a tropical curve $\Gamma$ in $M_g^{trop}$,
\item an effective divisor $D$ on $\Gamma$ of degree $d$, and
\item an isomorphism between $G'$ and the unique minimal semistable model of $\Gamma$ whose vertices contain the support of $D$.
\end{itemize}
The automorphism group $\Aut(G',D)$ acts on $\sigma_{G'}=\R_{\geq 0}^{E(G')}$ by permuting the entries of the vectors accordingly and the natural map $\mathring{\sigma}_{G'}\rightarrow \Div_{g,d}^{trop}$ factors through the injection $\mathring{\sigma}_{G'}/\Aut(G',D)\hookrightarrow \Div_{g,d}^{trop}$. Thus the set $\Div_{g,d}^{trop}$ arises as a colimit of the diagram $J_{g,d}\rightarrow \mathbf{RPC}^{face}$ and therefore carries the structure of a generalized cone complex. 

Finally, a maximally degenerate object $(G',D)$ in $J_{g,d}$ with all vertex weights equal to zero, has precisely $3g-3+d$ finite edges. Therefore the dimension of every maximal cone in $\Div_{g,d}^{trop}$ is $3g-3+d$. 
\end{proof}

\begin{remark}\label{remark_DivGammad=polyhedral}
Let $\Gamma$ be a stable tropical curve. The set $\Div_d^+(\Gamma)$ of effective divisors of degree $d$ on $\Gamma$ admits a natural rational polyhedral subdivision that arises by subdividing $\Sym^d\Gamma$ along the folds.

Let $G$ be the unique minimal stable model of $\Gamma$ and denote by $J_{g,d}/G$ the category of triples consisting of a stable pair $(G',D)$ in $J_{g,d}$ and an isomorphism between the stabilization of $G'$ with $G$. This datum induces a functor $J_{g,d}/G\rightarrow \mathbf{RPC}^{face}$ given by $G'\mapsto \sigma_{G'}$ together with a natural transformation to the constant functor $G\mapsto \sigma_G$. The set $\Div_d^+(\Gamma)$ is the preimage of the point $[\Gamma]\in\sigma_G$ in the colimit of all~$\sigma_{G'}$.
\end{remark}
\par
\begin{remark}\label{remark_stackybutok}
Using the language of tropical moduli stacks developed in \cite{CavalieriChanUlirschWise_tropstack}, we may consider the natural moduli functor $\calDiv_{g,d}^{trop}$ that associates to a rational polyhedral cone $\sigma$ the groupoid of semistable pairs $(G',D)$ together with non-zero edge length on $G'$ in the dual monoid $S_\sigma$. The proof of Proposition \ref{prop_Divgd=generalizedconecomplex} actually shows that $\calDiv_{g,d}^{trop}$ is representable by a \emph{cone stack} in the sense of \cite{CavalieriChanUlirschWise_tropstack} (see \cite[Section 3.4]{CavalieriChanUlirschWise_tropstack} for an analogous argument for the moduli stack $\calM_{g,n}^{trop}$ of $n$-marked stable tropical curves of genus $g$).

A stable tropical curve $\Gamma$ of genus $g$ (with real edge lengths) then corresponds to a morphism $\R_{\geq 0}\rightarrow \calM_g^{trop}$. Expanding on \cite[Section 4]{CavalieriChanUlirschWise_tropstack}, one can show that pullback $\R_{\geq 0}\times_{\calM_g^{trop}}\calDiv_{g,d}^{trop}$ is representable by a \emph{cone space} and its fiber over $1\in \R_{\geq 0}$ is exactly the polyhedral decomposition of $\Div_d^+(\Gamma)$ we have considered in Remark  \ref{remark_DivGammad=polyhedral}.
\end{remark}

\begin{remark}\label{remark_tropHassett} Given a tropical curve $\Gamma$ with $n$ legs, a \emph{marking} is an ordering $l_1<\cdots<l_n$ of the legs of $\Gamma$. Let $g\geq 0$ and $\calA=(a_1,\ldots, a_n)\in (0,1]\cap\QQ^n$ such that $2g-2+a_1+\cdots +a_n>0$. Following \cite{CavalieriHampeMarkwigRanganathan_tropHassett, Ulirsch_tropHassett}, we say that a marked tropical curve $\Gamma$ is \emph{stable of weight~$\calA$} if we have 
\begin{equation*}
2h(v)-2+\vert v\vert_E+\vert v\vert_\calA > 0
\end{equation*}
for every vertex $v\in V(G)$. Here $\vert v\vert_E$ denotes the \emph{inner valency} of the vertex $v$, i.e. the number finite edges emanating from $v$ (counting loops twice), and $\vert v\vert_\calA$ denotes the sum $a_{i_1}+\ldots a_{i_k}$, where $l_{i_1},\ldots,l_{i_k}$ are the legs emanating from $v$. 

Denote by $M_{g,\calA}^{trop}$ the space of tropical curves of genus $g$ with $n$ marked legs that are stable of type $\calA$. It has the structure of a generalized cone complex by \cite[Section 3]{Ulirsch_tropHassett}. If $\calA=(\epsilon,\ldots, \epsilon)$ with $\epsilon\leq\frac{1}{n}$, there is a natural morphism
\begin{equation*}\begin{split}
M_{g,\epsilon^d}^{trop}&\longrightarrow \Div_{g,d}^{trop}\\
(\Gamma, l_1<\ldots< l_n)& \longmapsto \big( \Gamma, \sum_{i=1}^n l_i\big)
\end{split}\end{equation*}
of generalized cone complexes, where $\sum_{i=1}^n l_i$ denotes the sum over the vertices the legs are emanating from. 
This map induces a homeomorphism
\begin{equation*}
M_{g,\epsilon^d}^{trop}/S_d\xrightarrow{\sim} \Div_{g,d}^{trop}
\end{equation*}
of the underlying topological spaces. 
\end{remark}

%% file: sec_tropdivisors.tex
\section{Specialization versus tropicalization}
\label{section_tropdivisor}

Let $k$ be an algebraically closed field~$k$ endowed with the trivial absolute value. 
Denote by $\calDiv_{g,d}^\an$ the non-Archimedean analytification of the moduli space of smooth curves together with an effective divisor of degree $d$. In this section we define a natural tropicalization map 
\begin{equation*}
\trop_{g,d}\,:\, \calDiv_{g,d}^{an}\longrightarrow \Div_{g,d}^{trop}
\end{equation*}
from the Berkovich analytic space $\calDiv_{g,d}^{an}$ to the tropical moduli space $\Div_{g,d}^{trop}$ and show that this map can be identified with a natural strong deformation retraction onto the non-Archimedean skeleton of $\calDiv_{g,d}^{an}$ in the sense of \cite{Thuillier_toroidal}. 

Note that for every algebraic stack $\mathcal{X}$ which is locally of finite type over $k$, there is an associated analytic stack $\mathcal{X}^{\mathrm an}$ defined by pullback with respect to the usual analytification functor on $k$-schemes, see \cite[Definition 2.18]{Ulirsch_tropisquot}. We usually abuse notation and denote by $\mathcal{X}^{\mathrm an}$ the associated topological space as defined in \cite[Definition 3.3]{Ulirsch_tropisquot}. Hence, if $\mathcal{X}$ is a separated algebraic Deligne-Mumford stack, by \cite[Proposition~3.8]{Ulirsch_tropisquot}, the space $\mathcal{X}^{\mathrm an}$ can be identified with  the Berkovich analytification of the coarse moduli space associated to $\mathcal{X}$. 

\subsection{Tropicalization of $\calDiv_{g,d}$} \label{sec:specialization}
We begin by defining the tropicalization map 
\begin{equation*}
\trop_{g,d}\,:\, \calDiv_{g,d}^{an}\longrightarrow \Div_{g,d}^{trop} \ .
\end{equation*}
A point in $\calDiv_{g,d}^\an$ is represented by a proper, smooth algebraic curve~$X$ of genus~$g$ over a field~$K$ that is a non-Archimedean extension of~$k$ together with an effective divisor $D$ on $X$ of degree $d$. Possibly after replacing~$K$ by a finite extension, there is a semistable model $\cXX/S$ of~$X$ over the spectrum~$S$ of the  valuation ring~$R$ of~$K$ together with a relative effective divisor $\calD$ on $\calX$ that does not meet the singularities in the special fiber $\calX_s$ and makes the divisor $K_\calX+\calD$ relatively ample.
Here we use that the moduli stack $\calDivbar_{g,d}$ is proper.
Its special fiber $(\calX_s, \calD_s)$ (as a Cartier divisor) is an element in $\calDivbar_{g,d}(\widetilde{K})$, where $\widetilde{K}$ denotes the residue field of $R$. On the level of points, the tropicalization map 
\begin{equation*}
\trop_{g,d}: \calDiv_{g,d}^\an \longrightarrow \calDiv_{g,d}^\trop 
\end{equation*}
associates to the pair $(X,D)$ the \emph{dual tropical curve}~$\Gamma$ of
 $\calX_s$ together with the \emph{specialization} of $D$ to $\Gamma$, an effective divisor of degree $d$ on $\Gamma$. More precisely, 
the weighted dual graph $G'$ is the incidence graph of $\calX_s$ 
together with the vertex weights $h(v)$ given by
the genus of (the normalization of) the corresponding irreducible component  of~$\calX_s$. The \emph{dual tropical curve} of $\calX$ is the tropical curve with semistable model $G'$ for which the {\em length} $|e|$ of the edge $e \in E(\Gamma)$ is defined to be $\val_R(f)$, where $xy=f$ is the local equation of the node corresponding to~$e$, and where~$\val_R$
denotes the valuation. For $v \in V(\Gamma)$ we denote the normalization of the component $C_v$ of the special fiber $\calX_s$ by $\widetilde{C}_v$. The {\em specialization of $D$ to $\Gamma$} is then defined as the multidegree 
\be
\mdeg(\calD_s) \= \sum_{v \in V(\Gamma)} \deg(\calD_s|_{\widetilde{C}_v}) \cdot [v]
\ee
of the special fiber $\calD_s$ of $\calD$, thought of as a divisor on $\Gamma$ (with support contained in the vertices of $G'$).
\par
The independence of the choices made in this construction is checked 
in~\cite{viv} for the moduli space of stable curves. It also follows 
a posteriori from Theorem~\ref{thm_tropDivgd} below.
\par
\begin{remark}
Instead of working with $(\cXX,\calD)$ as above, we can also work (at least in the case that~$S$ is the spectrum of a discrete
valuation ring) with any semistable model $\widetilde{\cXX}/S$ in which $\calD$ does not meet the singularities of $\calX_s$. The dual graph of the special fiber, metrized as above, is {\em equivalent} to that of~$\cXX/S$ in that it results from a subdivision of edges by a finite number of $2$-valent genus zero vertices. 
\end{remark}
\par
\subsection{Baker's specialization map}\label{section_Bakerspecialization}
 Let $X$ be a smooth curve over a non-Archi\-medean extension $K$ of $k$, and let $D$ be an effective divisor on $X$. Then there is a minimal skeleton $\Gamma$ associated to $X^{\an}$, see  \cite[Section 4.3]{Berkovich_book}, which is a deformation retract of $X^{\an}$, that is, there exists a continuous retraction map $\tau: X^{an}\rightarrow \Gamma$. Denote by $\tau: X(\overline{K}) \rightarrow \Gamma$ its restriction to $X(\overline{K})\subseteq X^{an}$. By linear extension, and using that $\Div_d(X_{\overline{K}}) = \Div_d(X(\overline{K}))$, we may define a homomorphism
\begin{equation*}\begin{split}
\tau_\ast: \Div_d(X_{\overline{K}})&\longrightarrow\Div_d(\Gamma) \\
\sum_ia_i p_i&\longmapsto \sum_i a_i \tau(p_i) \ .
\end{split}\end{equation*}
We also refer the reader to \cite[Section~2C]{bakerSpec} and, in particular, to \cite[Section~6.3]{BakerJensen} for details on this construction.
\par
Suppose now that $D$ is effective. Since $\calDivbar_{g,d}$ is proper, we may find a semistable model~$\calX$ of~$X$ over~$R$ (possibly after replacing $R$ by a finite extension and~$X$ and~$D$ by their base changes) such that the closure $\calD$ of $D$ on $\calX$ does not meet the singularities of the special fiber $\calX_s$ of~$\calX$ (and so that $K_\calX+\calD$ is relatively ample). If the support of $D$ is $K$-rational and $R$ is a discrete valuation ring, then 
the image $\tau_\ast (D)$  coincides with the multidegree of~$\calD_s$ on~$\calX_s$, thought of as a divisor on~$\Gamma$. In other words, 
\begin{equation*}
\trop_{g,d}\big([X],D\big)=\big(\trop_g([X]),\tau_\ast (D)\big) \ .
\end{equation*}
This observation has originally appeared in \cite[Remark 2.12]{bakerSpec} (also see \cite[Section 6.3]{BakerJensen}).

\subsection{The retraction to  the skeleton}\label{sec:retraction}

Let $X_0\hookrightarrow X$ be a \emph{toroidal embedding}, i.e. an open immersion of normal schemes locally of finite type over~$k$ that \'etale locally on $X$ admits an \'etale morphism $\gamma: X\rightarrow Z$ into a $T$-toric variety~$Z$ such that $\gamma^{-1}(T)=X_0$. Moreover, suppose for notational simplicity that $X$ is proper over $k$. 

In \cite{Thuillier_toroidal} Thuillier has constructed a strong deformation retraction 
\begin{equation*}
\bfp_{X_0\hookrightarrow X}: X_0^{an}\longrightarrow X_0^{an}
\end{equation*}
onto a closed subset $\frakS(X_0\hookrightarrow X)$ of~$X_0^{an}$ with the structure of a generalized cone complex, the \emph{non-Archimedean skeleton} of~$X_0$ (defined with respect to the toroidal compactification~$X_0\hookrightarrow X$). We refer the reader to \cite[Section 6]{acp} for a generalization of this construction to separated toroidal Deligne-Mumford stacks. 

In fact, Thuillier's construction in \cite{Thuillier_toroidal} extends to the analytification of the toroidal compactification $X^{an}$ and we obtain a strong deformation retraction to a compactified skeleton $\overline{\frakS}_{X_0\hookrightarrow X}$ of $X^{an}$. The resulting strong deformation retraction is proper and closed as a map $X^{an}\rightarrow \overline{\frakS}_{X_0\hookrightarrow X}$, since $\overline{\frakS}_{X_0\hookrightarrow X}$ is compact. Therefore $\bfp_{X_0\hookrightarrow X}$ is proper and closed as well.

By Theorem \ref{thm_Divbar}, the open immersion $\calDiv_{g,d}\hookrightarrow \calDivbar_{g,d}$ is toroidal and hence there is a natural strong deformation retraction 
\begin{equation*}
\bfp_{g,d}: \calDiv_{g,d}^{an}\longrightarrow\calDiv_{g,d}^{an}
\end{equation*}
onto the skeleton $\frakS_{g,d}$ of $\calDiv_{g,d}^{an}$. Together with the following Theorem \ref{thm_tropDivgd}, this implies Theorem \ref{mainthm_tropDiv} from the introduction.

\begin{theorem}\label{thm_tropDivgd}
There is a natural isomorphism $\Phi_{g,d}: \Div_{g,d}^{trop}\xrightarrow{\sim}\frakS_{g,d}$ that makes the diagram 
\begin{center}
\begin{tikzcd}
\calDiv_{g,d}^{an} \arrow[rd, "\bfp_{g,d}"'] \arrow[rr, "\trop_{g,d}"] &  & \Div_{g,d}^{trop} \arrow[ld,"\Phi_{g,d}", "\cong"'] \\
& \frakS_{g,d} &
\end{tikzcd}
\end{center}
commute. 
\end{theorem}
\par
Theorem \ref{thm_tropDivgd} in particular shows that the tropicalization map $\trop_{g,d}$ defined above is well-defined, continuous, proper, and closed.
\par
\begin{proof}[Proof of Theorem \ref{thm_tropDivgd}]
This is yet another version of the main result of \cite{acp}. We begin
by showing that the skeleton $\frakS_{g,d}$ is naturally isomorphic
to~$\Div_{g,d}^{trop}$. 
\par
Since $\calDiv_{g,d}\hookrightarrow \calDivbar_{g,d}$ is a toroidal embedding
by Theorem~\ref{thm_Divbar}, the stack $\calDivbar_{g,d}$ admits a natural
stratification by locally closed substacks (as in \cite[Section~3.1]{Thuillier_toroidal}). The locally closed strata are precisely the locally closed substacks parametrizing pairs $(X',D)$ for which the pair consisting of its weighted dual graph and the multidegree of $D$ is constant. Consequently, there is
a stratum $\calM_{(G',D)}$ for every object $(G',D)$ of $J_{g,d}$
(as in the proof of Proposition~\ref{prop_Divgd=generalizedconecomplex}). 
\par
For a pair $(G',D)$ in $J_{g,d}$ set 
\begin{equation*}
\widetilde{\calDiv}_{(G',D)}\=\prod_{v\in V(G')}\calDiv_{h(v),\vert v\vert, d_v}
\end{equation*}
where $\vert v\vert$ denotes the valency of $v$ in $G'$ and $d_v=D(v)$ denotes the degree of $D$ at the vertex $v$. Here we make use of the stack $\calDiv_{g,n,d}$ of effective divisors over $\calM_{g,n}$, which is constructed as the $d$-fold symmetric product of the universal curve over $\calM_{g,n}$, i.e.\ 
as the relative coarse moduli space of the morphism 
$\big[\calM_{g,1^n,\epsilon^d}/S_d\big]\rightarrow \calM_g$. 

We denote by $\Lambda^+_{(G',D)}$ the monoid of effective Cartier divisors on $\calDivbar_{g,d}$ whose support is contained in the closure of $\calDiv_{(G',D)}$. In other words, $\Lambda^+_{(G',D)}$ is the characteristic monoid of the natural divisorial logarithmic structure on $\calDivbar_{g,d}$ at the generic point of $\calDiv_{(G',D)}$. Notice that $\Lambda_{(G',D)}^+\simeq\N^{E(G')}$, since $E(G')$ is precisely the set of nodes of an element in $\calDiv_{(G',D)}$.

We observe:

\begin{itemize}
\item A stratum $\calM_{(G_1',D_1)}$ is in the closure of another stratum $\calM_{(G_2',D_2)}$ if and only if there is a weighted edge contraction $\pi: G_1'\rightarrow G_2'$ such that $\pi_\ast D_1=D_2$. In this case, the \'etale specialization $\Lambda_{(G_1',D_1)}^+\rightarrow \Lambda_{(G_2',D_2)}$ is given by the projection $\N^{E(G_1')}\rightarrow \N^{E(G_2')}$ and the induced map 
\begin{equation*}
\R_{\geq 0}^{E(G_2')}=\Hom\big(\Lambda_{(G_2',D_2)}^+,\R_{\geq 0}\big)\longrightarrow \Hom\big(\Lambda_{(G_1',D_1)},\R_{\geq 0}\big)=\R_{\geq 0}^{E(G_1')}
\end{equation*}
is precisely the face morphism induced from $\pi: G_1'\rightarrow G_2'$.

\item As in \cite[Proposition 3.4.1]{acp}, we have an equivalence  
\begin{equation*}
\calDiv_{(G',D)}\,\simeq\,\big[\widetilde{\calDiv}_{(G',D)}/\Aut(G',D)\big]
\end{equation*}
and therefore the action of the \'etale fundamental group of $\calDiv_{(G',D)}$ on $\R_{\geq 0}^{E(G')}=\Hom(\Lambda_{(G',D)}^+,\R_{\geq 0})$ surjectively factors through the action of $\Aut(G',D)$ that permutes the entries. 
\end{itemize}

The above observations together with \cite[Proposition 6.2.6]{acp} show that the non-Archimedean skeleton $\frakS_{g,d}$ arises as the colimit 
\begin{equation*}
\varinjlim_{(G',D)\in J_{g,d}}\Hom(\Lambda_{(G',D)}^+,\R_{\geq 0})=\varinjlim_{(G',D)\in J_{g,d}}\R_{\geq 0}^{E(G')}=\varinjlim_{(G',D)\in J_{g,d}} \sigma_{(G',D)}
\end{equation*}
and therefore $\frakS_{g,d}$ is naturally isomorphic to $Div_{g,d}^{trop}$ (as a generalized cone complex). 

Let us now show that above diagram commutes, i.e. that $\bfp_{g,d}=\Phi_{g,d}\circ \trop_{g,d}$. A point $x$ in $\calDiv_{g,d}^{an}$ gives rise to a morphism $S=\Spec R\rightarrow \calDivbar_{g,d}$ and the image of the closed point lies exactly in a stratum $\calDiv_{(G',D)}$. Thus the stable pair $(G',D)$ is the underlying combinatorial type of both $\bfp_{g,d}(x)$ and $\trop_{g,d}(x)$ and so they are in the interior of the same cone $\sigma_{(G',D)}=\R_{\geq 0}^{E(G')}$ in $\frakS_{g,d}\simeq \Div_{g,d}^{trop}$.
 
The edge lengths (of the edges in $G'$) of both of the resulting tropical curves agree, since $\bfp_{g,d}$ is given by taking valuations of the elements in $\Lambda_{(G',D)}^+$ and these precisely correspond to the deformation parameters in the family $\calX$ over $S$, as $\calX$ is the pullback of the universal curve over $\calDivbar_{g,d}$ to $S$. 
\end{proof}

\begin{remark}
Let $0<\epsilon\leq\frac{1}{d}$. On the level of underlying topological spaces, the two analytic stacks $\big[\calM_{g,\epsilon^d}^{an}/S_d\big]$ and $\calDiv_{g,d}^{an}$ are homeomorphic, since they have the same coarse moduli space (using \cite[Proposition 3.9]{Ulirsch_tropisquot}), and in Remark \ref{remark_tropHassett} we have seen that $\Div_{g,d}^{trop}$ is naturally homeomorphic to the quotient $M_{g,\epsilon^d}^{\trop}/S_d$. By \cite[Theorem 1.2]{Ulirsch_tropHassett} we have a natural identification of the non-Archimedean skeleton of $\calM_{g,\epsilon^d}^{an}$ with the tropical moduli space $M_{g,\epsilon^d}^{trop}$ of weighted stable tropical curves. Therefore, since this identification is invariant under the $S_d$-operations on both sides, we may also reduce Theorem \ref{thm_tropDivgd} to this earlier result. 
\end{remark}

\begin{proposition}
The tropicalization naturally commutes with the forgetful map, i.e. we have a commutative diagram
\begin{equation*}\begin{CD}
\calDiv_{g,d}^{an}@>\trop_{g,d}>> \Div_{g,d}^{trop}\\
@VVV @VVV\\
\calM_{g}^{an} @>\trop_g >> M_{g}^{trop} \ .
\end{CD}\end{equation*}
\end{proposition}

\begin{proof}
This is in essence the argument in the proof of~\cite[Theorem 1.2.2]{acp}. Let $\calX$ be a stable degeneration of a smooth curve $X$ over a discretely valued non-Archimedean field $K$ and let $\calX'$ be a semistable degeneration of $X$ whose stabilization is equal to $\calX$. Let $xy=f$ be the local equation of a node corresponding to an edge $e$ of the dual tropical curve $\Gamma$ of $\calX$ and let $x_iy_i =f_i$ (for $i=1,\ldots, k$) be equations of the nodes of $\calX$ that lie above the node $xy=f$. Then we have $f=f_1\cdots f_k$ and thus $\val(f)=\val(f_1)+\cdots \val(f_k)$ and so the edges $e_i$ in the dual tropical curve $\Gamma'$ corresponding to $x_iy_i=f_i$ form a subdivision of $e$. 

This proves the commutativity of the above diagram for points that can be represented by a semistable family over a discrete valuation ring. Since these points are dense in $\calDiv_{g,d}^{an}$ and $\calM_{g}^{an}$ respectively, and both $\trop_g$ and $\trop_{g,d}$ are continuous maps, the commutativity of the above diagram follows.
\end{proof}

\begin{remark} Assume that $X$ is a smooth curve over a non-Archimedean and algebraically closed extension $K$ of $k$. Expanding on Section \ref{section_Bakerspecialization}, we may define a specialization map $\tau_\ast: \Div^+_d(X^{an})\rightarrow \Div_d^+(\Gamma)$ and, using an argument analogous to the one in the proof of Theorem \ref{thm_tropDivgd} one can show that $\tau_\ast$ has a natural section $\Phi_{X,d}: \Div_d^+(\Gamma)\rightarrow \Div^+_d(X^{an})$ making the composition
\begin{equation*}
\bfp_{X,d}=\Phi_{X,d}\circ \tau_\ast
\end{equation*}
into a strong deformation retraction onto a closed subset $\frakS_{d}(X)$ of $\Div^+_d(X^{an})$. If $\calX$ is a regular stable model of $X$ over $R$, corresponding to a morphism $S\rightarrow \calMbar_g$ with $S=\Spec R$,  the closed subset $\frakS_{d}(\calX)$ is the non-Archimedean skeleton of $\Div^+_d(X)^{an}$ in the sense of Berkovich \cite[Section 5]{Berkovich_contractibilityI} associated to the regular semistable model $S\times_{\calMbar_g}\calDivbar_{g,d}$ of $\Div_{d}^+(X)$. 
\end{remark}

%% file: sec_tropHodge.tex
\section{Tropicalizing the Hodge bundle } \label{sec:troptheHodge}\label{section_tropHodge}

From now one we specialize from general divisors to canonical divisors
and the Hodge bundle. Contrary to the case of algebraic curves, the
canonical linear system on a tropical curve~$\Gamma$ without legs comes with a
distinguished element
\begin{equation*}
K_\Gamma \= \sum_{v \in V(\Gamma)} (2h(v) + |v| -2)\cdot v
\end{equation*}
with support at the vertices of~$\Gamma$. We denote by $|K_\Gamma|$ the
canonical linear series. 
\par
In \cite{LinU} Lin and the second author introduce a tropical analogue of the Hodge bundle $\omoduli[g]$ and of its projectivization $\PP\omoduli[g]$ on the moduli space $M_{g}$. As a set, the \emph{tropical Hodge bundle} $\omoduli[g]^\trop$ is defined to be the set of isomorphism classes of pairs $(\Gamma,f)$ consisting of a stable tropical curve $\Gamma$ of genus $g$ and a rational function $ f\in \Rat(\Gamma)$ with $K_\Gamma + {\rm div}(f) \geq 0$. Its projectivization\footnote{This space was called 
$\calH_g^{trop}$ in \cite{LinU}, and $\omoduli[g]^\trop$ was called  $\Lambda_g^{trop}$ there. Here we mainly follow the notation conventions of \cite{bcggm}.} $\proj\omoduli[g]^\trop$
parametrizes pairs $(\Gamma,D)$ consisting of a stable tropical curve $\Gamma$ of genus $g$ and an effective divisor $D=K_\Gamma+\div(f)$ in $\vert K_\Gamma\vert$.
Both spaces come with a natural forgetful map to $\moduli[g]^\trop$.
\par
\begin{proposition}[\cite{LinU} Theorem~1]\label{prop_Hodgebundle} The tropical Hodge bundle $\PP\omoduli[g]^{trop}$ is a closed subset of $\Div_{g,2g-2}^{trop}$ that canonically carries the structure of a generalized cone complex of (maximal) dimension $5g-5$.
\end{proposition}

The Hodge bundle is not equidimensional, see Example~\ref{ex:dim} below.
Proposition~\ref{prop_Hodgebundle} shows in particular that $\PP\omoduli[g]^{trop}$ is a closed subset of $\Div_{g,2g-2}^{trop}$ that is a subcomplex of a subdivision of $\Div_{g,2g-2}^{trop}$. 

\begin{proof}[Proof of Proposition \ref{prop_Hodgebundle}]
Proposition \ref{prop_Hodgebundle} has already been proved as part of \cite[Theorem~1]{LinU} (and building upon the polyhedral description of tropical linear systems from \cite{GaKe, MiZh}). We rephrase the main insights of this proof using the language developed in Section \ref{section_modulitropdivisors}.

Let $G'$ be a semistable finite vertex-weighted graph of genus $g$ and an effective divisor $D\in\Div_{g,2g-2}(G')$ making $(G',D)$ into a stable pair. Write $G$ for the stabilization of $G'$. We will show that the pullback of the tropical Hodge bundle to $\sigma_{G'}=\R_{\geq 0}^{E(G')}$ is given by a finite union of linear subspaces of $\sigma_{G'}$. 

For simplicity choose an orientation on every edge of the graph $G$; the resulting structure will not depend on this choice. Consider a tropical curve $\Gamma$ whose underlying graph is $G'$. In order to specify a rational function $f$ on $\Gamma$ such that $D=K_\Gamma +\div(f)$ (up to a global additive $\R$-operation) we need to specify a collection of integers $(m_e)\in\Z^{E(G)}$ (one for each edge of the stabilization $G$ of $G'$), the initial slopes of $f$ at the origin of the edge $e$, subject to the condition
\begin{equation*}
2h(v)-2+\vert v\vert = \sum_{\textrm{outward edges at }v}m_e + \sum_{\textrm{inward edges at }v} -(\deg D\vert_e +m_e) \ .
\end{equation*}
Notice that by \cite[Lemma 1.8]{GaKe} there are, in fact, only finitely many choices for the initial slopes $m_e$.

In each of the finitely many cases that such an $f\in\Rat(\Gamma)$ exists, the continuity on $f$ imposes a collection of linear conditions on the coordinates of $\sigma_{G'}=\R_{\geq}^{E(G')}$ (i.e. the edge lengths of $G'$). The intersection of $\sigma_{G'}$ with such a linear subspace is a cone in the generalized cone complex structure on $\PP\omoduli[g]^{trop}$. 
\end{proof}

Let  $\PP\omoduli[g]^\an$ be the analytification of the projective Hodge bundle over an algebraically closed field $k$ endowed with the trivial absolute value. In this section we recall in detail the construction of the tropicalization map on the Hodge bundle from~\cite[Proposition~6]{LinU} and elaborate on its properties.

The moduli space $\PP\omoduli[g]^\an$ parametrizes pairs $(X/K,K_X)$ consisting of a point $X/K \in \moduli[g]^\an$ as recalled above, together with a divisor~$K_X$ that is equivalent to the canonical bundle $\omega_{X/K}$. We define a natural tropicalization map
\begin{equation*}
 \trop_\Omega: \PP\omoduli[g]^\an \to \proj\omoduli[g]^\trop 
\end{equation*}
by setting
\begin{equation} \label{eq:defTropOmega}
\trop_\Omega(X/K,K_X) \= \trop_{g,2g-2}(X/K,K_X) \ ,
\end{equation}
where $\trop_{g,2g-2}$ is the tropicalization map introduced in 
Section~\ref{sec:specialization}.
\par
\begin{Prop} \label{prop:tropomega}
The tropicalization map $\trop_\Omega$ is well-defined, continuous, proper, and closed.
\end{Prop} 
\par
\begin{proof} The fact that  $\trop_\Omega$ is well-defined can be shown using a moving lemma
as in \cite[Lemma 4.20]{bakerSpec}. We will see an alternative proof of this fact in Section \ref{section_mainresult} in the framework of this article.

Let $\calX$ be a semistable model of $X$ over a discrete valuation ring $R$ (or a finite extension thereof) extending $k$. We may assume that $\calX$ is regular; otherwise we blow up accordingly. By a moving lemma, such as \cite[Proposition 1.11]{Liu_AG}, we may find a (not necessarily effective) canonical divisor $K_\calX$ on $\calX$ that does not meet the singularities in the special fiber. It is well-known that the multidegree of $K_\calX$ in the special fiber is equal to $K_\Gamma$ (see e.g.\cite[Remark~4.18]{bakerSpec}). Any effective canonical divisor $K_X$ on $X$ is equivalent to the generic fiber of $K_\calX$ and therefore the specialization of $K_X$ to $\Gamma$ is equivalent to $K_\Gamma$.  

The discretely valued points in $\PP\omoduli[g]^{an}$ are dense and, since $\trop_{g,2g-2}$ is continuous and $\PP\omoduli[g]^{trop}$ is closed by Proposition 
\ref{prop_Hodgebundle} above, we obtain that $\trop_{g,2g-2}(x)$ is in $\PP\omoduli[g]^{trop}$ for every (not necessarily discretely valued) point $x\in \PP\omoduli[g]^{an}$. 
Since $\PP\omoduli[g]^{trop}$ is naturally a closed subset of $\Div_{g,2g-2}^{trop}$ the properness and closedness of $\trop_{\Omega}$ follow from the corresponding properties of $\trop_{g,2g-2}$.
\end{proof}
\par
\begin{definition} 
The \emph{realizability locus} $\PP\calR_\Omega$ in $\PP\omoduli[g]^{trop}$ is the image
\begin{equation*}
\PP\calR_\Omega=\trop_\Omega(\PP\omoduli[g]^\an)
\end{equation*}
of the tropicalization map.
\end{definition}
\par
The realizability locus $\PP\calR_\Omega$ is the locus of tuples $(\Gamma,D)$ consisting of a stable tropical curve $\Gamma$ and a canonical divisor $D$ for which there is a stable family $\calX$ of curve over a valuation ring $R$ together with an effective canonical divisor $K_X$ on the generic fiber $X$ of $\calX$ such that $\Gamma$ is the dual tropical curve of $\calX$ and $D$ is the specialization of $K_X$.

%% file: sec_twisteddifferentials.tex
\section{Twisted differentials and the global residue condition} \label{sec:twd}

From now on we work over the field $\CC$ of complex numbers. 
Let~$X/\CC$ be a smooth and proper algebraic curve, i.e.\ a compact
Riemann surface.
We let $\omoduli[g] \to \moduli[g]$ be the space of pairs
$(X,\omega)$ consisting of an algebraic curve together with a non-zero holomorphic 
one-form~$\omega$ on~$X$. This is the Hodge bundle over the
complex-analytic moduli space of curves, deprived of the zero section. 
The multiplicities of the zeros of~$\omega$ define a partition~$\mu$
of $2g-2$ and the subspaces  $\omoduli[g](\mu)$ with fixed partition~$\mu$
form a stratification of $\omoduli[g]$. We will focus most of the
time on the {\em principal stratum} corresponding to the partition
$\mu = (1,\ldots,1)$ and we usually put $n=|\mu|$.
\par
In this section we recall from~\cite{bcggm} the description of a 
compactification
of the strata of $\omoduli[g]$. More concretely, observe that there is a natural
map $\varphi: \omoduli[g](\mu) \to \moduli[g,{[\mu]}]$ sending $(X,\omega)$ to the
curve marked by the zeros of~$\omega$. Here $\moduli[g,{[\mu]}]$ is
the quotient of $\moduli[g,n]$ by the symmetric group that permutes the
entries of~$\mu$.
The main theorem of~\cite{bcggm} is a characterization
of the closure of $\varphi$. The version given here highlights the 
possible scaling parameters of one-parameter families approaching
the boundary. These scaling parameters will reflect the location of 
the support of the corresponding tropical divisors. We need to set up
some notation to recall the theorem for the case of holomorphic abelian
differentials.
\par
\begin{definition} \label{def:type}
We call any 
tuple $\mu=(m_1,\ldots,m_n)\in\ZZ^n$ such that $\sum m_i=2g-2$ and that 
$m_1\ge \ldots\ge m_r>m_{r+1}=0=\ldots=m_{r+s}>m_{r+s+1}\ge\ldots m_{r+s+p}$
a type. We denote by $p_1$ the number of $-1$ occuring in this tuple.
\end{definition}
\par
We assign a type to any meromorphic differential via the multiplicities
of its associated divisor.

The moduli space $\omoduli[g](\mu)$ parametrizes meromorphic one-forms
whose divisor is of type~$\mu$. We may view these spaces as strata
of a twisted Hodge bundle (see~\cite{bcggm}, but we do not need
this viewpoint here). 
\par
\subsection{Level graphs}
Let $\aG = (V,E)$ be an (unmetrized) graph. A {\em full order}~$\oG$ 
on~$\aG$ is an order $\succcurlyeq$
on the vertices~$V$  that is reflexive, transitive,
and such that for any $v_1, v_2 \in V$ at least one of the
statements $v_1 \succcurlyeq v_2$ or $v_2 \succcurlyeq v_1$ holds.
We call any function $\ell:V(\aG)\to\ZZ_{\le 0}$ such
that $\ell^{-1}(0)\ne\emptyset$ a {\em level function} on $\aG$.
Note that a level function induces a {\em full order} on $\aG$
by setting $v \preccurlyeq w$, if $\ell(v) \leq \ell(w)$.
A {\em level graph} is a graph together with a choice
of a level function. Abusing notation, we use the symbol~$\oG$ also
for level graphs.
\par
For a given level~$L$ we call the subgraph of $\lG$ that consists of all 
vertices $v$ with $\ell(v) > L$ along with edges between them the 
{\em graph above level $L$} of $\lG$, and denote it by
$\lG_{>L}$. We similarly define the graph $\lG_{\geq L}$
{\em above or at level $L$}, and the graph $\lG_{=L}$
{\em at level $L$}. An edge $e\in E(\lG)$ of a level graph $\lG$ is called 
{\em horizontal} if it connects two vertices of the same level, and it is 
called {\em vertical} otherwise. Given a vertical edge $e$, we denote 
by $v^+(e)$ (resp.~$v^-(e)$) the vertex that 
is its endpoint of higher (resp.~lower) level.

\subsection{Twisted differentials} \label{sec:twdif}
Let $C$ be a nodal, 
in general non-smooth curve over the complex numbers.
Let $\mu = (m_1,\ldots,m_n)$ be a type.
A {\em \twd of type $\mu$} on a stable $n$-pointed curve $(C,\bfs)$ is 
a collection of (possibly meromorphic)
differentials $\eta_v$ on the irreducible components~$C_v$ of $C$ such that 
no $\eta_v$ is identically zero with the following properties.
\begin{itemize}
\item[(0)] {\bf (Vanishing as prescribed)} Each differential $\eta_v$ is 
holomorphic and nonzero outside of the nodes and marked points of~$C_v$. 
Moreover, if a marked point $s_i$ lies on~$C_v$, then $\ord_{s_i} \eta_v=m_i$.
\item[(1)] {\bf (Matching orders)} For any node of $C$ that identifies 
$q_1 \in C_{v_1}$ with $q_2 \in C_{v_2}$, the vanishing orders satisfy
$\ord_{q_1} \eta_{v_1}+\ord_{q_2} \eta_{v_2}\=-2.$
\item[(2)] {\bf (Matching residues at simple poles, MRC)}  If at a node of $C$
  that identifies $q_1 \in C_{v_1}$ with $q_2 \in C_{v_2}$ the condition
  $\ord_{q_1}\eta_{v_1}=
\ord_{q_2} \eta_{v_2}=-1$ holds, then $\Res_{q_1}\eta_{v_1}+\Res_{q_2}\eta_{v_2}=0$.
\end{itemize}
Let $\aG$ be the dual graph of $C$. Recall that the vertices $v$ in $\Gamma$
correspond to the irreducible components $C_v$ of $C$. If~$\ell$ is a level
function on $\Gamma$, 
we write $C_{>L}$ for the subcurve of $C$ containing only the components $C_v$
with $v$ of level strictly bigger than $L$.
Similarly, we define $C_{=L}$. If two components $C_v$ and $C_w$ with
$\ell(v) <  \ell(w)$ intersect in the point $q$,
we denote  by $q^-$ the corresponding point on $C_v$, and we write
$v = v^-(e)$ for the edge~$e$ in~$\aG$ connecting $v$ and $w$.
\par
Denote by $\oG$ the full order on the dual graph $\aG$ given by a
level function. We say that a \twd $\eta$ of type $\mu$ on $C$ 
is called {\em compatible with~$\lG$}  if in addition it also 
satisfies the following two conditions.
\begin{itemize}
\item[(3)]{\bf (Partial order)} If a node of $C$  identifies 
$q_1 \in C_{v_1}$ with $q_2 \in C_{v_2}$, then $v_1\succcurlyeq  v_2$ if and 
only if $\ord_{q_1} \eta_{v_1}\ge -1$. Moreover,  $v_1\asymp v_2$ if and only if
$\ord_{q_1} \eta_{v_1} = -1$.
\item[(4)] {\bf (Global residue condition, GRC)} For every level $L$
and every connected component~$Y$ of $C_{>L}$ 
the following condition holds: Let
$q_1,\ldots,q_b$ denote the set of all nodes where~$Y$ intersects $C_{=L}$. Then
$$ \sum_{j=1}^b\Res_{q_j^-}\eta_{v^-(q_j)}=0,$$
where we recall that $q_j^-\in C_{=L}$ and $v^-(q_j)\in \lG_{=L}$.
\end{itemize}
\par

\subsection{The characterization of limit points} \label{sec:charlimit}

Suppose that $S$ is the spectrum of a discrete valuation ring~$R$ with residue 
field~$\CC$, whose maximal ideal is generated by~$t$. Let  $\cXX/S$ be a
family of semi-stable curves with smooth generic fiber~$X$ and special
fiber $C$. Let  $\omega$ be a section of $\omega_{\cXX/S}$ of type~$\mu$
whose divisor is given by the sections $\bfs = (s_1,\ldots,s_n)$
with multiplicity~$m_i$. The triple $(\cXX/S, \bfs, \omega)$
is called a \emph{pointed family of stable differentials}, if
moreover $(\cXX/S,\bfs)$ is stable. 
Then we define the {\em scaling factor $\ell(v)$} for the node~$v$ as
the non-positive integer such that the restriction of the
meromorphic differential $t^{-\ell(v)}\cdot \omega$ to
the component $C_v$  of the special fiber corresponding to~$v$
is a well-defined and generically non-zero differential~$\eta_v$ on~$C_v$
(see \cite[Lemma~4.1]{bcggm}).
The $\eta_v$ are called the \emph{scaling limits} of~$\omega$.
\par
\begin{Thm}[\cite{bcggm}] \label{thm:bcggm}
If $(\cXX/S, \bfs, \omega)$ is as above, 
then the function $\ell(v)$ defines a full order on the
dual graph $\Gamma$ of the special fiber of~$\cXX$ and the collection
$\eta_v|_{X_v}$ is a twisted differential of type~$\mu$ compatible with 
the level function~$\ell$.
\par
Conversely, suppose that $C$ is a stable $n$-pointed curve with dual
graph~$\Gamma$ and $\eta = \{\eta_v\}_{v \in V}$ is a twisted differential
of type~$\mu$ compatible with a full order~$\oG$ on~$\Gamma$. Then for
every level function $\ell:\Gamma \to \ZZ$ defining the full order~$\oG$ and
for every assignment of integers~$n_e$ to horizontal edges
there is a stable
family $\cXX/S$ over~$S = \Spec(\CC[[t]])$ with smooth generic fiber and
special fiber~$C$ that satisfies the following properties:
\begin{itemize}
\item[i)] There exists a global section~$\omega$ of $\omega_{\cXX/S}$ whose
horizontal divisor $\div_{\rm hor}(\omega) = \sum_{i=1}^n m_i \Sigma_i$
is of type~$\mu$ and whose scaling limits are the collection
$\{\eta_v\}_{v \in V}$.
\item[ii)] The intersections $\Sigma_i \cap C = \{s_i\}$ are smooth points of
the special fiber and $\eta$ has a zero of order~$m_i$ in $s_i$. 
\item[iii)] There exists a positive integer~$N$ such that a local equation
  near every node corresponding to  a
    horizontal edge~$e$ is $xy=t^{N n_e}$, and it is
$xy = t^{N(\ell(q^+(e)) - \ell(q^-(e)))}$ for every vertical edge~$e$.
\end{itemize}
\end{Thm}
\par
\begin{proof}
The first statement is the necessity of the Theorem~1.3
of~\cite{bcggm}, proven in Section~4.1. Note that the arguments
given in loc.\ cit.\ for this direction hold over any discrete valuation
ring.
\par
For the second statement
one has to trace the proof of sufficiency of this theorem, given
in Section~4.4 of loc.\ cit. As stated there (see equation~(4.8) and the last two paragraphs of
the proof of Addendum~4.8), there are no
constraints for the plumbing fixtures to be used for plumbing
horizontal nodes, whereas for the plumbing fixtures used for
every vertical nodes, given by an edge~$e$, the level function
$\ell_0$ used for plumbing has to satisfy the divisibility
constraint 
$$(\ord_{q^+(e)} \eta + 1) \,\,|\,\,(\ell_0(q^+(e)) - \ell_0(q^-(e)))\,.$$
Multiplying the prescribed function~$\ell$ by a sufficiently
divisible~$N$, the resulting level function~$\ell_0=N\cdot \ell$ 
satisfies this divisibility property.
\end{proof}

\subsection{Dimension and period coordinates} \label{sec:percoord}

In preparation for the dimension statements in Section~\ref{sec:dimresults}
we recall here two results about the geometry of strata
of meromorphic differentials. Consider the neighborhood of a point
$(X,\omega) \in \omoduli(\mu)$. We denote by $Z$ the $r+s$ zeros
and marked points of~$\omega$ and let $P$ be the $p$ poles of~$\omega$. 
On such a neighborhood, integration of the meromorphic
one-form against a basis of the relative cohomology group
$H^1(X \setminus P, Z; \ZZ)$ gives local coordinates,
called \emph{period coordinates}.
See \cite{kdiff} for a proof of this statement  (including the case
of $k$-differentials) and for references to the history of this result.
The fact that these functions are local coordinates 
also proves the following dimension statement.
\par
\begin{Thm} \label{thm:dim}
The stratum $\omoduli(\mu)$ has dimension $2g-1 + n$ if
the type  $\mu$ is holomorphic (i.e.\ if $p=0$),  and it has dimension $2g-2 + n$
if the type $\mu$  is strictly meromorphic (i.e.\ if $p>0$)
\end{Thm}
\par
The following result is the special case for one-forms
of a main result of \cite{chencomplete}.
\par
\begin{Thm}[\cite{chencomplete}] \label{thm:nocomplete}
The projectivisation of a stratum $\PP\omoduli(\mu)$ of strictly
meromorphic type (i.e.\ with $p>0$) does not contain a complete curve.
\end{Thm}

\subsection{The image of the residue map}

Since the global residue condition imposes strong constraints
on the residues, we need a criterion on which residues can actually
be realized. Since a twisted differential is a collection of
meromorphic differentials, rather than just holomorphic differentials,
we have to deal more generally with types of meromorphic differentials.
Recall from the beginning of Section~\ref{sec:twd} the conventions used to
denote types of meromorphic differentials. In particular, $p_1 \leq p$
denotes the number of simple poles. For every type~$\mu$ with $p \neq 0$ we let 
$$ \Res\,:\, \omoduli(\mu) \to H$$
be the residue map, whose range is contained by the residue theorem in
$$H \= \Bigl\{\bfx = (x_1,\ldots,x_p) \in \CC^p\,:\, \sum_{i=1}^p x_i= 0 \Bigr\}\,.$$ 
Moreover we define the 'non-zero set' $N \subset H$ to consist
of those~$\bfx$ with $x_i \neq 0$ whenever $m_i = -1$. By definition
of a stratum, the image of $\Res$ is obviously contained in~$H \cap N$.
\par
To illustrate the problem of determining the image of~$\Res$, consider
differentials~$\eta$ of type  $\mu=(a,-b,b-2-a)$ with $a \geq 0$, $b \geq 2$
and $b-a-2 \leq -2$ on a projective line with coordinate~$z$.
We may assume that the zero of~$\omega$ is at $z=1$, while the poles are
at $z=0$ and $z=\infty$. Consequently, $\eta = C (z-1)^a dz/z^b$ with
$C \neq 0$. This implies that the residue is non-zero, in fact
${\rm Im}(\Res) = H \setminus \{(0,0)\} \subset H$.
The image of the map $\Res$ in the general case was determined by Gendron
and Tahar in~\cite{GenTah}. We restate a simplified version of their main
result that is sufficient for our purposes.
\par
\begin{Prop} \label{prop:ressurj} 
\begin{enumerate}[(i)]
\item If $g \geq 1$ then $\Res$ is surjective
onto $N \cap H$, which is a non-empty set unless $p=p_1=1$.
\item If $g=0$ and $p>p_1>0$ or $p_1=0$ and there does not exist an index
$1 \leq i \leq n$ with
\be \label{eq:micond}
m_i > \Bigl(\sum_{j=r+s+1}^{r+s+p}  -m_j \Bigr) - p -1
\ee
then $\Res$ is surjective onto $N \cap H$.
\item If $g=0$ and $p_1=0$ and there exists an index $1 \leq i \leq n$
with~\eqref{eq:micond}, then $\Res$ is surjective onto
$N \cap H \setminus \{(0,\ldots,0)\}$.
\item
If $g=0$ and $p=p_1=2$ then $\Res$ is surjective onto
$N \cap H$. 
\item If $g=0$ and $p=p_1>2$, then the
image of $\Res$ contains all tuples in $N \cap H$ consisting
of~$p$ $\RR$-linearly independent vectors.
\end{enumerate}\end{Prop}
\par
\begin{proof} All the conditions except for the case $g=0$ and $p=p_1$ (only
simply poles) are restatements of the case of abelian differentials
in \cite[Theorem~1.1-1.5]{GenTah}. The case $p=p_1=2$
follows directly from the preceding discussion. The case $p=p_1>2$
is stated for $r=1$ in \cite[Proposition~1.6]{GenTah}. For $r>1$ it
can be easily deduced from this by the procedure of splitting a zero,
or it can be derived from the more involved formulation in \cite[Proposition~1.7]{GenTah}.
\end{proof}

%% file: sec_CharDim.tex
\section{The realizability locus} \label{sec:descR}

In this section we prove our main Theorem~\ref{thm:realizecrit} describing
the realizability locus of tropical canonical divisors over an algebraically
closed field of characteristic zero.4
For tropical curves with rational sides lengths and canonical
divisors with rational coordinates we use the complex-analytic
techniques of the previous section to characterize the image
of the tropicalization map. For general points the results
follow from the general properties of $\trop_\Omega$.

\subsection{From rational functions to enhanced level graphs}\label{section_mainresult}

To formalize the correspondence between rational functions and
level graphs, we introduce the following enhancement of the
notion of a level graph. 
We consider vertex weighted graphs of the form $\Gamma = (V,E,L,h)$ \emph{without edge lengths},
where $(V,E)$ is a classical graph, $L$ is a finite set of legs, i.e. infinite half-edges starting at a vertex, and $h$ is a marking of the vertices with integer numbers. 
After choosing a level function on the underlying graph $(V,E)$, we call $\oG$ a level graph. 
We divide every edge $e$ in two half-edges and write $\Lambda $ for the set containing all half-edges and all legs in $\Gamma$.
Note that every $\lambda \in \Lambda$ is adjacent to a unique vertex.
An  {\em enhanced level graph ~$\eG$} is a level graph ~$\oG$  
as above together with an assignment $k: \Lambda \to \ZZ$ such
that the following compatibility conditions hold.
\par
\begin{itemize}
\item[i)] If $e^+$ and $e^-$ are two half edges  forming an edge~$e$, 
then $k(e^+) + k(e^-) = -2$. An edge is horizontal if 
and only if $k(e^\pm) = -1$ for both its half-edges. Moreover,
if $e$ is a vertical edge consisting of the half-edges~$e^+$ 
leading downwards and $e^-$ leading upwards, then $k(e^+) > k(e^-)$.
\item[ii)] For each vertex $v \in V$ 
\be \label{eq:genussum}
\sum k(\lambda) \= 2h(v) -2\,,
\ee  
where the sum is over all $\lambda \in \Lambda$ which are adjacent to~$v$.
\end{itemize}
Let $v$ be a vertex in an enhanced level graph~$\eG$. Then we define the
type~$\mu(v)$ as the ordered tuple consisting of all~$k(\lambda)$,
where~$\lambda$ is a half-edge adjacent to~$v$. Note that~$\mu(v)$ is a
type in the sense of Definition~\ref{def:type}, if we replace $g$ by $h(v)$.
\par
This definition is motivated from the notion of twisted differentials.
In fact, given a stable curve~$X$, let  $\aG$ be the associated graph consisting of the dual graph of $X$ with legs attached for each marked point and $h$ given by the genera of the irreducible components $X_v$ for $v \in V$.
Note that up to the choice of a metric this is the same construction as in Section~\ref{section_tropdivisor}. Every $\lambda$ in $\Lambda$ gives rise to a point $z_\lambda$ in $X_v$, where $v$ is the vertex adjacent to $\lambda$: If $\lambda$ is a half-edge, we let $z_\lambda$ be the node corresponding to the edge containing $\lambda$, and if $\lambda$ is a leg, we let $z_\lambda$ be the associated smooth point. For any \twd~$\eta$ on $X$, we  decorate all $\lambda \in \Lambda$  with the order of  $\eta$ at $z_\lambda$.
This defines an enhanced level graph structure on $\Gamma$.
\par
\begin{Lemma} \label{le:fgivesLev}
Let $\Gamma$ be a tropical curve. To every element
$D = K_\Gamma + \div(f) \in |K_\Gamma|$ we can associate a natural structure
of an enhanced level graph~$\Gamma^+ = \Gamma^+(f)$ on some
realization of~$\Gamma$.
\end{Lemma}
\par
\begin{proof} Let $\Gamma_0$ be the minimal realization of~$\Gamma$ subdivided 
with vertices at the places where~$f$ is not differentiable.
We use the function~$f$ itself to give~$\Gamma_0$ a full order
i.e.\ for nodes $v,w$ of $G_0$ we declare $v \succcurlyeq w$ if and only if
$f(v) \geq f(w)$. 
\par
We provide each vertex~$v$ of~$\Gamma_0$ with 
$2h(v) - 2 + \sum_e (1 + s(e))$ 
legs, each given the  decoration $k=1$.
Here the sum runs over all non-leg half-edges adjacent to~$v$ and $s(e)$ 
denotes the slope of~$f$ along~$e$, oriented to be pointing away
from~$v$. The fact that  $D \in |K_\Gamma|$ is equivalent to this number 
of legs being indeed non-negative for all vertices. We provide
each half-edge~$e$ which is not a leg, with $k = -s(e)-1$,
using the same orientation convention. The conditions for an
enhanced level graph now follow immediately.
\end{proof}
\par
We need two more notions to state our main theorem.
\par
\begin{Defi}
A vertex~$v$ of an enhanced level graph is called~\emph{inconvenient}
if $h(v)=0$ and if its type $\mu(v) = (m_1,\ldots, m_n)$ has the property
that $p_1 = 0$  and there exists an index~$i$ such that the
inequality~\eqref{eq:micond} holds.
\end{Defi}
\par
A cycle is called \emph{simple}, if it does not visit any vertex more
than once. Recall the tropicalization map
$\trop_\Omega: \PP\omoduli[g]^\an \to \proj\Lambda_g^\trop $
from Proposition~\ref{prop:tropomega}. 
\par
\begin{Thm}\label{thm:realizecrit}
Suppose that $k$ is an algebraically closed field of characteristic zero. A pair $(\Gamma, D)$ with $D = K_\Gamma + \div(f)$ in the tropical canonical
linear series lies in the image of $\trop_\Omega$ if and only if 
the following two conditions hold: 
\begin{itemize}
\item[i)] For every inconvenient vertex~$v$ of $\Gamma^+(f)$
there is a simple cycle $\gamma \subset \Gamma$ based at~$v$ 
that does not pass through any node at a level smaller than~ $f(v)$.
\item[ii)] For every horizontal edge~$e$ there is a simple cycle 
$\gamma \subset \Gamma$ passing though~$e$ which
does not pass through any node at a level smaller than $f(e)$.
\end{itemize}
\end{Thm}
\par
\input{pic_graph_3}
\par
Figure~\ref{cap:IllEdge} illustrates the conditions of the theorem. The
value of~$f$ is given by the height of the point in $\Gamma^+(f)$ over its
image point in~$\Gamma$. In this example there is
a simple cycle through the horizontal edge in the foreground. However
all the possible simple cycles through this edge pass through
the vertex with two markings, which is at a lower level. Consequently,
this graph is not realizable. Note also that realizability
depends on the edge lengths here: If the edge containing
the vertex with two markings were shorter (and all the other
lengths remained the same), the corresponding vertex could be at a level
above the horizontal edge and the corresponding divisor would be realizable.
\par
\begin{proof}[Proof of Theorem \ref{thm:realizecrit}]
Let us first assume that $k = \mathbb{C}$.
The conditions i) and ii) define a closed subset of $\PP\omoduli[g]^{trop}$. Tropical curves~$\Gamma$ with rational edge lengths and divisors $D = K_\Gamma + \div(f)$ associated to a function~$f$ are dense in this subset. Since $\trop_\Omega$ is continuous and closed by Proposition~\ref{prop:tropomega}, the image of $\trop_\Omega$ is the closure of this locus. We may therefore assume that $\Gamma$ has rational edge lengths. Moreover, if $(\Gamma, K_\Gamma+\div(f))$ is realizable, we may rescale the edge lengths by a global constant and still obtain a realizable object in $\PP\omoduli[g]^\trop$. Therefore we may assume that $\Gamma$ has integral edge lengths.
\par
Suppose that the enhanced level graph $\Gamma^+(f)$ associated with~$D$
satisfies i) and ii) and the integrality hypothesis made above on~$\Gamma$. We want to show that there is a twisted differential
of type $\mu = (1,1,\ldots,1)$ on a stable
pointed curve~$C$ with dual graph $\Gamma$, compatible
with the enhanced level structure~$\Gamma^+(f)$ and then apply
the 'converse' implication in Theorem~\ref{thm:bcggm}. 
For every vertex~$v$ this amounts to finding a differential of type $\mu(v)$ 
on some smooth curve~$C_v$.
This is indeed the type of a meromorphic differential on~$C_v$ by property~ii)
of an enhanced level graph. The matching order condition and the partial
order condition of a twisted differential are built into the 
condition~i) of an enhanced level graph. 
\par
Hence the main point is to choose the curves~$C_v$ and the
differential~$\eta_v$ such that the matching residue condition (MRC)
and the GRC can be satisfied. For this purpose, we want to apply 
Proposition~\ref{prop:ressurj}. By the following procedure
we specify residues which on the one hand lie in $H \cap N$ at every
node and satisfy both MRC and GRC, and which on the other hand 
are non-zero at inconvenient nodes, and match the conditions of the
last item in Proposition~\ref{prop:ressurj} at nodes with only simple poles.
\par
Take the cycles $\gamma_i, i \in I_t$  corresponding to inconvenient vertices
$v_i$ by condition~i) and the cycles $\delta_j, j \in J_h$  corresponding to
horizontal edges~$e_j$ by condition~ii) and provide them with some orientation.
Let $\{\alpha_i , i \in I_t\} \cup \{\beta_j, j \in J_h\}$ be a
collection of complex numbers such that no sum of a subset of the $\pm \alpha_i$
and $\pm \beta_j$ is real. (I.e.\ the $\alpha_i$ and $\beta_j$ 
lie in a complement of a finite union
of real codimension one hyperplanes in $\CC^{|I_t|+|J_h|}$.)
Starting from residue zero at each edge we increase, for every $i \in I_t$, 
the prescribed residue at all the half-edges~$e$ with
$f(e) = f(v_i)$ by $\alpha_i$ if the outward pointing orientation of~$e$
agrees with the orientation of the cycle and by $-\alpha_i$ otherwise. 
For every $j \in J_h$ 
we increase the prescribed residue at all the half-edges~$e$ with
$f(e) = f(e_j)$ by $\beta_j$. The collection of residues prescribed
in this way is non-zero for every horizontal edge (by the choice
of the cycles $\delta_j$ and since the choice of the $\alpha_i$ and $\beta_j$
avoids unintended cancellations), it is non-zero at every inconvenient
node (by the choice of the cycles $\gamma_i$) and satisfies the
residue theorem (since a cycle enters and exits any vertex the
same number of times), i.e.\ the prescribed residues lie in 
the image of the residue map at each inconvenient node
by Proposition~\ref{prop:ressurj}. At each of the vertices with
only simple poles, i.e.\ with $ p = p_1$,
the residues are non-zero and $\RR$-linearly independent
(if there is more than one pair of such poles, i.e.\ if $p_1>2$ at
such a vertex). Consequently, by Proposition~\ref{prop:ressurj}, the
residues lie in the image of the residue map at each vertex.
Finally, we check that the GRC continues to hold at each step of adding
the contributions along a cycle $\gamma_i$ or $\delta_j$. We give the
details for the first case, the second being the same, replacing $f(v_i)$
by $f(e_j)$ everywhere. In fact, the addition procedure prescribes a zero
total sum of residues to each of the component of $C_{>f(v_i)}$ the cycle
passes through, so the GRC holds for $C_{\geq f(v_i)}$. Since the cycle does
not pass through levels below $f(v_i)$, the GRC for those levels remains
valid. If~$w$ is a vertex with level $f(w) > f(v_i)$ then all the edges to
level~$f(v_i)$ are unseen in the GRC for $C_{\geq f(w)}$ and hence the GRC
for these levels continues to hold, too.
\par
Consequently, we can now use Theorem~\ref{thm:bcggm} with the level
function $\ell = f$ and with $n_e = |e|$, the length in $\Gamma$ for any
horizontal edge. The conclusion of the theorem is precisely that
the divisor $D + \div(f)$ is the specialization of an effective
canonical divisor on a graph equivalent to~$\Gamma$, with all the 
lengths rescaled by the integer~$N$ of Theorem~\ref{thm:bcggm}~iii).
Hence $(\Gamma,D)$ lies in the image of $\trop_\Omega$ by 
definition of this map in Equation~\eqref{eq:defTropOmega}.
\par 
\medskip
Let us now show the converse implication. Points of the form $(X,D)$, 
where $X$ is the smooth generic fiber of a stable curve $\cXX$ over the 
valuation ring $R$ of a finite extension of $\CC(t)$ are dense
in $\proj \omoduli[g]^\an$. Since $\trop_\Omega$ is continuous by 
Proposition~\ref{prop:tropomega},  it suffices to show that $\trop_\Omega(X,D)$
satisfies  conditions i) and ii) in our claim. Denote by~$S$ the spectrum 
of~$R$. Moreover, let $\omega$ be a stable differential
on $\cXX$ such that the divisor of its generic fiber is $D$. 
We may assume, by the density of the principal stratum,
that ${\rm div}(\omega)^{\rm hor} = \sum_{i=1}^{2g-2} s_i(S)$ consists of
$2g-2$ images of sections.
\par
Let $\ell$ be the level function on the dual graph~$\Gamma$
of the special fiber given by the scaling parameters of this family
(cf.\ Section~\ref{sec:charlimit}).
Let $\eta$ be the \twd on the special fiber~$C$ of $\cXX$, obtained as the
scaling limit of $\omega$, and let $\Gamma^+(\ell)$ be the enhanced
level graph given by~$\ell$ and the enhancement given by~$\eta$, 
as described before Lemma~\ref{le:fgivesLev}. We want to show that 
$\trop_\Omega(X,D)  = (\Gamma, K_\Gamma + {\rm div}(\ell))$
that the enhanced level graph~$\Gamma^+(\ell)$ satisfies i) and~ii). 
\par
Concerning the first claim, we note
that by definition of $\eta$ exactly $2h(v)-2 - \sum k(e)$ sections
of~ $\omega$ lie in the irreducible component $C_v$ of the special fiber
associated to the vertex~$v$, counted with multiplicity.
Here the sum runs over all half-edges adjacent to~$v$ which are not
legs at~$v$. 
By definition of $\trop_\Omega$ in~\eqref{eq:defTropOmega} the first claim thus
amounts to showing that the slope defined by~$\ell$ of non-leg half edges~$e$
is equal to~$- k(e) -1$. This in turn is a consequence of the way degenerating
families are built by plumbing (cf.\ \cite{bcggm} Theorem~4.5). The core
observation is that at a node $xy=t^a$ the differential on the two
ends of the node is $(x^k+t^{a(k+1)}\tfrac{r}x)dx$ and
$-t^{a(k+1)}\cdot (y^{-k-2} + \tfrac{r}y) dy$, where~$t^{a(k+1)}r$ is the period
of~$\omega$ along the vanishing cycle. Consequently, the difference of
scaling parameters (or, equivalently, values of the level function) 
is $a(k+1)$, proving the slope claim. 
\par
To show i), we work with the complex topology. Note that there exists a disc
$\Delta$ in $\CC$ such that $\cXX(\CC)$ can be extended to a complex analytic space
$\cXX_\Delta(\CC)$ over $\Delta$. At every inconvenient vertex~$v$ the restriction
of~$\eta$ to~$v$ has non-zero residue at $q^-(e)$ for some edge ~$e$
with $k(e)<0$  by Proposition~\ref{prop:ressurj}. To illustrate the idea
for constructing the necessary cycle, let $e_1,\ldots,e_m$ be the edges adjacent to~$v$
with $k(e)<0$. Choose a continuously varying family $\beta_j(s)$ for $s \in \Delta$ 
of simple closed curves in the fibers $\cXX_s $ belonging to  the homotopy class
 which is pinched to the node~$e_j$. If all the curves $\beta_j(s)$
are separating for one (hence every) $s \in \Delta$, the period
of~$\omega_s$ around~$\alpha_0$ is zero by Stokes' theorem. This gives
a contradiction to the non-zero residue in the limit. Consequently, 
there is some non-trivial cycle~$\gamma$ passing through~$v$.
\par
To deal with the general case and to derive the claimed property of~$\gamma$
we revisit the proof of the GRC, compare \cite{bcggm} Section~4.1. 
Let $A = \alpha_1(s) \cup \ldots \cup \alpha_m(s)$ be the union of simple
closed curves which are pinched (when $s \to 0)$ to the nodes
joining a level $\geq \ell(v)$ to a level $<\ell(v)$.
If $\beta_j(s)$ is separating the connected component
of $X_s \setminus A$ it lies in, we obtain the same contradiction
from Stokes' theorem as 
before. In fact, let $I \subset \{0,\ldots,s\}$ be the index set of
curves bounding a component of $X_s \setminus (A \cup \beta_j)$. Then
$$ \int_{\beta_j} t^{\ell(v)} \omega(t)  \,+\,
\sum_{i \in I} \int_{\alpha_i} t^{\ell(v)} \omega(t) \= 0 $$
by Stokes. The first term of this sum tends to the residue we are
interested in. The other terms tend to the residue of the
limiting twisted differential at level~$\geq \ell$ at a node
corresponding to an edge to level $<\ell$, which is zero since the
limiting differential is holomorphic there (see the condition
'partial order' in the definition of twisted differential
in Section~\ref{sec:twdif}).
Consequently, if some $\beta_j(s)$ is not
separating its connected component in $X_s \setminus A$, there exists a
cycle as claimed in~i). The argument for horizontal edges is the same
and gives~ii).
\medskip
\par
Hence we have proved the theorem in the case $k = \CC$. If  $K \subset L$ 
is a field extension of two trivially valued algebraically closed fields of characteristic zero, we have a natural surjective projection map $(\PP\omoduli[g,L])^\an \rightarrow  (\PP\omoduli[g,K])^\an$ which is compatible with the tropicalization map, i.e. the diagram 
\begin{equation*}
\begin{tikzcd}
 (\PP\omoduli[g,L])^\an \arrow[dr,"\trop_{\Omega}"]\arrow[d]  &    \\
  (\PP\omoduli[g,K])^\an \arrow{r}[swap]{\trop_{\Omega}}  & \proj\omoduli[g]^\trop \\ 
\end{tikzcd}
\end{equation*}
is commutative. Hence the realizability locus does not depend on the choice 
of the algebraically closed ground field of characteristic zero, which implies 
our claim.
\end{proof}
\par
\input{pic_graph_1}
\par
\begin{Ex} \label{ex:dumbbell}
{\rm Figure~\ref{cap:DumbOK} shows points in the realizability locus
over the dumbbell graph in genus two, i.e.\ all vertex genera zero.
Those points in the realizability locus consist
of two symmetrically placed marked
points on either dumbbell end (left picture) or double point
anywhere on the central edge (including the ends) of the dumbbell
(picture on the right).
A canonical divisor whose support consists of two different points
on the central edge of the dumbbell (see Figure~\ref{cap:DumbBAD})
is not in the realizability locus, since the edge between those
two points is horizontal in the enhanced level graph and separating,
thus violating condition~ii) in Theorem~\ref{thm:realizecrit}.
These two figures are reinterpretation of the corresponding figures
in~\cite{GaKe} in our viewpoint of level graphs.
\input{pic_graph_2}
}\end{Ex} 
\par

\begin{proof}[Proof of Proposition~\ref{prop:tropomega}, alternative
proof that the image of $\trop_\Omega$ belongs to $|K_\Gamma|$]
We need to show that for any given graph~$\Gamma$ there exists
a rational function $f \in \Rat(\Gamma)$ with $K_\Gamma + \div(f) \geq 0$ 
such that $\Gamma^+(f)$ satisfies conditions i) and ii) of the 
preceding theorem. We prove this by induction on the genus. Since
adding marked points and increasing the vertex genus can only improve
the situation concerning inconvenient nodes, it suffices to treat
the case that all vertex genera $h(v) = 0$.
\par
For $g=2$, there are two cases. For the graph with three nodes
joining the two vertices (and in general: for any graph $\Gamma$
without separating edges), the canonical divisor $K_\Gamma$ is in the
image of $\trop_\Omega$. For the dumbbell graph, we take a function~$f$
that is constant on the edges and has a global minimum on the separating
edge, see the picture on the right hand side of Figure~\ref{cap:DumbOK}.
\par
In the induction step, we consider a graph~$\Gamma$ of genus~$g$
and remove a non-separating edge~$e$. Let $\Delta = \Gamma -e$ be the resulting
graph. There are two cases to consider. First, suppose that the two
ends of~$e$ are different nodes in~$\Delta$. Then $\Delta$ is
semistable and we start with the $f_0$ given by induction on
the stable graph equivalent to~$\Delta$. We complete this to a
function~$f$ on~$\Gamma$ having slope $\geq -1$ on each half-edge of~$e$.
(This is possible for all values of~$f_0$ at the ends of~$e$.)
Together with the induction hypothesis this condition implies 
$K_\Gamma + \div(f) \geq 0$.
Neither a horizontal separating edge nor a trivalent vertex with
negative decorations has been added, hence conditions i) and ii)
continue to hold.
\par
Second, suppose that~$e$ is a cycle, adjacent to some vertex~$v$. 

If~$\Delta=\Gamma-e$ is semi-stable, we simply declare $f$ to be constant on~$e$.
Otherwise, there is a separating edge~$e_s$ ending at~$v$ and
$\Delta \setminus e_s$ is semi-stable. In this case we take~$f_0$
from $\Delta \setminus e_s$ by induction and complete it to~$f$
constant on~$e$ and with slopes~$-1$ on the two half-edges of~$e_s$
(i.e.\ $\div(f)$ contains twice the midpoint of $e_s$). The
conditions i), and ii) and $K_\Gamma + \div(f) \geq 0$ follow from 
the construction.
\end{proof}
\par
\input{pic_graph_4}
\input{pic_graph_6}
\input{pic_graph_5}
\begin{Ex} \label{example_K4}{\rm To give a more involved example we discuss the realizability
locus over the complete graph~$K_4$. We first claim that there are
five types of maximal dimensional cones, as given
in Figure~\ref{cap:K4realize1}, Figure~\ref{cap:K4realize2} and
Figure~\ref{cap:K4realize3}.
\par
To prove this claim, we establish some notation. Suppose that $\Gamma^+$
is the enhanced level graph corresponding to a canonical divisor 
in the realizability locus.
Let $v_1,\ldots,v_4$ be the four vertices of the original $K_4$ and
let $w_1,\ldots,w_n$ with $n \leq 4$ be the remaining vertices of~$\Gamma$,
each of them having at least one leg. Consider the vertices on the top
level of~$\Gamma^+$. A vertex $w_i$ cannot be on top level, since its
two non-leg half edges have $k=-1$, since it has at least one leg
and since the sum of decorations is equal to~$-2$. Suppose one of the
$v_i$ on top level is decorated with a leg. This requires $v_i$ to have
(at least) three horizontal edges, otherwise the sum of decorations
cannot be~$-2$. Together with the previous argument this implies that all~$v_i$
lie on top level, and that each of them  is decorated with a single leg.
By Lemma~\ref{le:dimcompute} below the cone with this configuration has
dimension~$6$, strictly less than the maximal dimension~$8$.
\par
Consequently the top level consists of vertices $v_i$ without legs,
hence each of them is adjacent to precisely two horizontal edges.
Thus the subgraph on top level is a simple cyle.
In the $K_4$ the length of the cycle might be three or four. In the case
of a $4$-cycle on top, the case of a single $w_i$
with just one leg on one of the edges is ruled out by the sum  of
decorations being equal to~$-2$. The configuration in
Figure~\ref{cap:K4realize3} remains and attains the maximal dimension.
\par
Suppose the top level is a $3$-cycle consisting of $v_1,v_2,v_3$ and
that $v_4$ is on lower level. Consider the edges $e_i$ for $i=4,5,6$
joining $v_{i-3}$ to $v_4$. For each partition $(4,0,0)$, $(3,1,0)$, $(2,2,0)$
and $(2,1,1)$ of the four vertices~$w_i$ on these three edges there is a
unique solution to the enhancement conditions, leading to the graphs in
Figure~\ref{cap:K4realize1} and Figure~\ref{cap:K4realize2}, all of them
of maximal dimension. Graphs with less than four~$w_i$ are degenerations
hereof, and hence of strictly smaller dimension.
\par
Note that for a given tropical curve~$\Gamma$ with underlying graph~$K_4$ the preimage~$\pi^{-1}\big([\Gamma]\big)$ does not meet all of these maximal cones. For example the graph in Figure~\ref{cap:K4realize3}
is possible for any edge lengths, with the canonical divisor supported on
an arbitrary pair of disjoint edges. However, the graph pictured
in Figure~\ref{cap:K4realize1} on the left, with the canonical
divisor supported on a pair of adjacent edges is possible if and
only if $|e_6| < |e_4|$, $|e_6| < |e_5|$.
\par
The realizability locus is connected in codimension one over the closure
of the $K_4$-cone. To see this, note that contracting one the of
horizontal edges on the top level $4$-cycle in Figure~\ref{cap:K4realize3}
and reopening it as a vertical edge, connects this cone to the
cone in Figure~\ref{cap:K4realize1}~(left). This cone is connected to
the cone in Figure~\ref{cap:K4realize2}~(left) by pushing one of the $w_i$
adjacent to~$v_4$ into~$v_4$. This cone is connected to the cone in
Figure~\ref{cap:K4realize2}~(right) by pushing the isolated $w_i$ on~$e_6$
through~$v_4$ onto the edge~$e_5$. Finally, the cone in
Figure~\ref{cap:K4realize2}~(left) is also connected to the cone
Figure~\ref{cap:K4realize1}~(right) by pushing the vertex $w_i$ on~$e_5$
adjacent to~$v_4$ through~$v_4$.
}\end{Ex}
\par
It is an interesting combinatorial question  whether in general the
maximal cones of the realizability locus are connected in codimension one.

\subsection{Dimensions} \label{sec:dimresults}

The fundamental Theorem of Bieri-Groves \cite[Theorem A]{BieriGroves} (also see \cite[Theorem~2.2.3]{EinsiedlerKapranovLind}) shows that, given a closed subvariety of a split algebraic torus, its tropicalization admits the structure of a polyhedral complex of the same dimension. In our situation, the tropical Hodge bundle does not admit a natural embedding into a toric variety, but rather  a toroidal embedding in the sense of \cite{KKMSD_toroidal}, the compactification $\calDivbar_{g,2g-2}$ of $\calDiv_{g,2g-2}$ over $\calMbar_g$. In this situation a weaker version of the Bieri-Groves Theorem holds (see \cite[Theorem~1.1]{Ulirsch_tropcomp}) and we only know that the realizability locus (i.e. the tropicalization of $\PP\omoduli[g]$) is a generalized cone complex of dimension $\leq 4g-4$. This result technically only applies in the case when the boundary has no self-intersection, but the arguments immediately generalize to our situation. Our methods allow us to prove the following much stronger statement.
\par
\begin{Thm} \label{thm:equi}
The realizability locus $\PP \cRR_\Omega$
is a generalized cone complex, all whose maximal cones have dimension~$4g-4$. 
\par
The fiber in $\PP \cRR_\Omega$ over a maximal-dimensional cone $\sigma_{G}$ 
in $M_{g}^\trop$ (i.e. for a trivalent graph $G$ with all vertex-weights 
$h(v)=0$) is a generalized 
cone complex, all whose maximal cones have relative dimension~$g-1$. 
\end{Thm}
\par
Recall that in Figure \ref{cap:IllEdge} we have seen that the realizability locus is not a subcomplex of~$\calDiv_{g,2g-2}^{trop}$.
\par
\begin{Ex} \label{ex:dim}
{\rm We revisit Example~\ref{ex:dumbbell}. The dumbbell graph is one of
the two trivalent genus two graphs. For any edge lengths
assigned to the dumbbell, the fiber of $\PP\omoduli[g]^\trop$
over the corresponding tropical curve is the folded square with
two ends pictured in Figure~\ref{cap:dimdumb}. 
\input{pic_dim_dumbbell} 
The realizability locus corresponds to the thickened line segments,
drawn horizontally. Notice here that the canonical divisor $K_\Gamma$ (which
corresponds to the third corner in the triangle) is not in the realizability
locus.
}\end{Ex}
\par
The dimension estimates are based on the following lemma.
The contraction procedure in the lemma stems from the fact that
the length information encoded in those genus zero nodes is not
recorded when passing to the associated tropical curves with divisor. Note
that for all enhanced level graphs that appear in Theorem~\ref{thm:realizecrit},
i.e.\ those resulting from Lemma~\ref{le:fgivesLev}, we
have $\Gamma^+ = \Gamma^+_0$ in the following statement.
\par
\begin{Lemma} \label{le:dimcompute}
For every level graph $\Gamma^+$ let $\Gamma^+_0$ be the graph obtained by
successively contracting edges in $\Gamma^+$ that have a $(n+1)$-valent
genus zero node with $n \geq 1 $~marked points at one of its ends. 
The dimension  of a cone $\sigma(\Gamma^+)$ in the realizability locus with
associated level graph $\Gamma^+$ is equal to
one less than the number of levels ${\rm Lev}(\Gamma^+_0)$
plus the number of horizontal edges $E_H(\Gamma^+_0)$, i.e.\ 
\bes \label{eq:dimformula}
\dim(\sigma(\Gamma^+)) \= |{\rm Lev}(\Gamma^+_0)| - 1 + |E_H(\Gamma^+_0)|\,.
\ees
\end{Lemma}
\par
\begin{proof}
Assign a real number $d_i \leq 0$ ('depth') to each
level $i \in {\rm Lev}(G,\ell)$, such that $d_0 = 0$ and $d_i < d_j$
if $i <j$. Then endow any edge~$e$ joining the vertices $v_1$ and $v_2$ with
$\ell(v_1) > \ell(v_2)$ with length $(d_{\ell(v_1)} - d_{\ell(v_2)})/(k(e^+)+1)$
and endow horizontal edges with arbitrary lengths.  By the construction
this tropical curve admits a unique continous function~$f$ (up to addition
of a global constant) that is linear of slope zero on horizontal edges
and linear of slope $-k(e^+)-1$ on each edge (as viewed from the top end).
This implies that $\dim(\Gamma^+) \geq |{\rm Lev}(\Gamma^+_0)| - 1
+ |E_H(\Gamma^+_0)|$.
\par
On the other hand, every rational function~$f$ on a tropical curve with enhanced
level graph $\Gamma^+_0$ determines uniquely a collection of real numbers $d_i$
with $d_0=0$ and $d_{\ell(v_1)} - d_{\ell(v_2)} = |e|s(e)$ whenever
$\ell(v_1) > \ell(v_2)$. This implies the converse estimate.
\end{proof}
\par
\begin{proof}[Proof of Theorem~\ref{thm:equi}]
To prove the upper bound~$4g-4$, we compare with the complex dimension of the
moduli space of twisted differentials (denoted by $\frakM^{\rm ab}(\oG)$
in \cite{kdiff}) compatible with a level graph~$\oG$. (The dimension
does not depend on the enhancement.) 
Each level contributes at least one to the dimension of $\frakM^{\rm ab}(\oG)$,
namely by rescaling the differentials on that level by a scalar and
definition the dimension of $\frakM^{\rm ab}(\oG)$ is the sum 
of the dimensions of the spaces of twisted differentials at each level.
Consequently, the maximal dimension is bounded above by the number of horizontal
edges plus $\dim_\CC \frakM^{\rm ab}(\oG)-1$. This sum is computed
in \cite[Theorem~6.1]{kdiff} to be equal to $\dim_\CC \omoduli[g](\mu)-1$, 
where~$\mu$ is the type of the twisted differential. This quantity
is maximized for the principal stratum $\mu = (1, \ldots, 1)$ and gives
$\dim_\CC \omoduli[g](\mu)-1  = 2g-2+|\mu| = 4g-4$ by Theorem~\ref{thm:dim}
and thus the claimed upper bound.
\par
To show that this upper bound is always attained we have to split
vertices whose contribution to $\frakM^{\rm ab}(\oG)$
is greater than one. The claim follows from the more precise
statement in the subsequent proposition.
\end{proof}
\par
\begin{Prop} \label{prop:maxvert}
Maximal dimensional cones of the realizability locus
correspond precisely to
the enhanced level graphs~$\eG$ with the following properties.
\begin{itemize}
\item[(i)] All the vertices have vertex genus zero.
\item[(ii)] Each vertex is either
\begin{itemize}
\item[(ii.1)] $n$-valent ($n \geq 3$) with precisely two edges which are
legs or edges to lower level.
\item[(ii.2)] $n$-valent ($n \geq 3$) with precisely one edge which is
a leg or an edge to lower level.
\end{itemize}
\item[(iii)] Each level~$L$ contains either
\begin{itemize}
\item[(iii.1)] precisely one vertex as in (ii.1). All the edges of this
vertex to higher level disconnect the subgraph $\Gamma^+_{\geq L}$, or
\item[(iii.2)] only vertices as in (ii.2). At each of these nodes~$v$
moreover $|v|-2$ edges disconnect  the subgraph $\Gamma^+_{\geq L}$
while the remaining edges of the nodes at level~$\ell$
together with the connected components of the subgraph $\Gamma^+_{> L}$
form a simple cycle
\end{itemize}
\end{itemize}
\end{Prop}
\par
In (iii.2) the valence $|v|$ refers to the valence of the subgraph
$\Gamma^+_{\geq L}$, i.e.\ an edge to lower level does not contribute. 
A \emph{simple cycle} is a graph isomorphic to some $n$-gon.
In the proof we will see that the
conditions in the theorem can 
be explained using period coordinates (see Section~\ref{sec:percoord}).
\par
\begin{proof}
In order to show that these cones are maximal we need to show
that the contribution of each level to the dimension of $\frakM^{\rm ab}(\oG)$
is at most one. Then we conclude using \cite[Theorem~6.1]{kdiff}, since
then the number of levels has to be at least $4g-4-|E_H(\ol{\Gamma})|$.
\par
We start by discussing the dimension contribution for the cones 
in the statement of the proposition. Recall from Theorem~\ref{thm:dim}
that the space of differentials corresponding to a $n$-valent vertex
of genus zero has dimension $n-2$.
\par
If there are two marked points (or edges to lower level)
as in $(ii.1)$ this space is parametrized by the relative
period between the two marked points and $n-1$
residues. Moreover the condition in $(iii.1)$ and the GRC implies that
all the residues are zero and the resulting contribution of that
level~$L$ is of dimension one.
\par
If there is only one marked point (or edges to lower level), the
space of differentials is parametrized by the $n-1$ residues with one
constraint given by the
residue theorem. The condition in~(iii.2) and the GRC imply again that
$n-3$ residues vanish. Hence each node contributes again individually
one to the complex dimension of the space of differentials at that level.
Moreover, the cycle constraint in~(iii.2) implies that this residue is the
same for each vertex at the given level. Consequently the total contribution
of that level to the space of twisted differentials is one, as claimed.
\par
To show that the cones listed in the proposition are the only
cones of maximal dimension, we show that we can 
split the vertices in an enhanced
level graph until the condition of the proposition are met, while
maintaining the conditions Theorem~\ref{thm:realizecrit}~(i) and~(ii)
of the realizability locus along the splitting procedure. E.g.\ 
while some vertex genus is positive, we apply the splitting
of Figure~\ref{cap:deg1} where $a_i \geq 0$  and where $b_i \leq -1$.
\input{pic_deg_1.tex}
We may thus assume from now on that all vertex genera are zero.
In order to show that the nodes at a given level~$i$ can be split 
until their contribution to the dimension of $\frakM^{\rm ab}(\oG)$
is one (in the sense computed in the beginning of the proof)
we could argue combinatorially but arguing geometrically as
follows seems more enlightening. Consider the space of twisted
differentials~$\eta= \{\eta_v\}$ compatible with the level graph
currently under consideration. Suppose that the subspace of twisted
differentials with all $\eta_v$ fixed except for those with $\ell(v)=i$
has dimension greater than one. To put it differently, we assume that the
projectivisation of this subspace has positive dimension. This subspace
is cut out inside $\PP\omoduli(\mu)$ by a collection of residue conditions.
Since $\PP\omoduli(\mu)$ does not contain a projective curve
by Theorem~\ref{thm:nocomplete}, a subspace defined by residue conditions
does not contain such a curve either. Consequently, there 
is some way to degenerate the meromorphic differential, thus increasing
either the number of levels or the number of horizontal edges. This
also increases the dimension of the corresponding cone in the
realizability locus and we can repeat the process until each
cone has dimension one, with the caveat given in Lemma~\ref{le:dimcompute}
that trees of marked points are contracted.
If we replace each such tree as in Figure~\ref{cap:deg2}
\input{pic_deg_2.tex}
then this is also
a degeneration of the graph where all trees of marked points
are contracted, the number of levels and horizontal edges is
the same and this graph satisfies the conditions of
Theorem~\ref{thm:realizecrit}
if the graph prior to the replacement did. This concludes the
proof of the existence of a splitting procedure.
\end{proof}

\subsection{The realizability locus for strata of abelian differentials}
\label{sec:Rstrata}

Let $\mu$ be a partition of $d$. We say that an effective divisor $D$ of degree $d$ on a tropical curve $\Gamma$ has type $\mu$ if the multiplicities at its support define the partition $\mu$. Notice, in particular, that, in complete analogy with the situation for the projective algebraic Hodge bundle~$\PP\omoduli[g]$, the tropical Hodge bundle~$\PP\omoduli[g]^{trop}$  admits a stratification by strata $\PP\omoduli[g]^{trop}(\mu)$ that are indexed by partitions $\mu$ of $2g-2$. 
\par
Theorem~\ref{thm:realizecrit} contains also a characterization of the
\emph{realizability locus $\PP \cRR_\Omega(\mu)$ of the stratum
  of type~$\mu$}, defined as the image of $\trop_\Omega |_{\PP \omoduli[g](\mu)^{an}}$
of the restriction of the tropicalization map to the corresponding
stratum $\PP \omoduli[g](\mu)^{an}$ of abelian differentials. In fact,
the proof of our main theorem applies verbatim to give the
following criterion.
\par
\begin{Prop}
An element $D = K_\Gamma + \div(f)$ in the tropical canonical
linear series lies in $\PP \cRR_\Omega(\mu)$ if and only if
$D$ is a divisor of type~$\mu$ and for the enhanced level graph $\Gamma^+(f)$
the two conditions i) and ii) of Theorem~\ref{thm:realizecrit} hold. \end{Prop}
\par
\begin{Ex} \label{ex:dumbbell2}
{\rm We give for example the realizability locus $\PP \cRR_\Omega(2)$
if the underlying graph is the dumbbell graph, i.e.\ as a subset
of Example~\ref{ex:dumbbell}. Restricted to this graph, the tropical
Hodge bundle is disconnected and consists of (isolated) double zeros
at the midpoint of either of the dumbbell cycles and of a one-dimensional
component with double zero on the central edge of the dumbbell.
The midpoints of the dumbbell cycles satisfy the criteria i) and ii).
A point of multiplicity two in the interior of the the central edge
(as Figure~\ref{cap:DumbOK} on the right) does not satisfy criterion~i), since
a vertex with one zero (of order two) and two poles (of order two)
is inconvenient). However, if the double is located on the vertices
of the dumbbell, the vertices are no longer inconvenient and the
criteria are satisfied.
\par
In conclusion, the realizability locus $\PP \cRR_\Omega(2)$ on the
dumbbell graph consists of four 'Weierstrass' points as
in Figure~\ref{cap:realOm2}.
\input{pic_4_dumbbell.tex}
}\end{Ex}

\subsection{Algorithmic aspects}\label{sec:algo}

Our main theorem can be turned into an algorithm
to compute the simplicial structure of the realizability locus.
\begin{enumerate}[i)]
\item For each genus~$g$ construct the finitely many abstract
graphs $G=(E,V,L)$ with a genus function $h: V \to \ZZ_{>0}$
of genus~$g$ (in the sense of~\eqref{eq:genus}) and with $|L| = 2g-2$
that are stable.
\item There is a finite number of partial orders on~$G$ such
that any two vertices joined by an edge are comparable (equality
permitted). For each of those partial orders there is a finite
number of enhancements $k$ with the properties that $k(e^\pm)=-1$
for both half edges of $e$ joining $v_1$ and $v_2$ with $v_1 \asymp v_2$
and such that whenever there is an edge~$e$ joining $v_1$ and $v_2$ with
$v_1 \succ v_2$ then $k(e_1) \geq 0$, where $e_1$ is the half-edge of~$e$
adjacent to $v_1$.
\par
To see this, we attribute $k(e^\pm)=-1$ to all edges with $v_1 \asymp v_2$
and argue inductively top-down: for each vertex~$v$ such that
all the upward pointing half-edges have already been assigned an
enhancement, there is a finite number of possibilities to assign
a non-negative enhancement to each of  the downward pointing half-edges~$e^+$
such that the genus formula~\eqref{eq:genussum} holds. We complete
this enhancement on each of the complementary half edges~$e^-$
using the condition $k(e^+) + k(e^-) = -2$ and proceed to another vertex.
\item For each of the partial orders there is a finite number of full
orders that refine the partial order.
We assume for notational convenience that the full order is given
by the level function~$\ell$.
\item For each of the horizontal edges~$e$ (i.e.\ both half edges are
decorated with $k(e^\pm)=-1$) check if~$e$ disconnects the
graph $G_{\geq \ell(e)}$ and discard the graph if this is the case.
Here $\ell(e) := \ell(v)$ for any of the two vertices adjacent to~$e$.
\item Using the enhancement we can determine the set of inconvenient
vertices $I \subset V$. For each vertex $v \in I$ check if~$v$ disconnects
the graph $G_{\geq \ell(v)}$ and discard the graph if this is the case. 
\item The realizability locus consists of a cone $\sigma=\sigma_{(G,h,k,\ell)}$
for each tuple $(G,h,k,\ell)$ not discarded. The cone $\sigma_{(G,h,k,\ell)}$
parametrizes the following tropical curves. Assign
as in the proof of Lemma~\ref{le:dimcompute} a real number
$d_i \leq 0$ ('depth') to each level $i \in {\rm Lev}(G,\ell)$,
such that $d_0 = 0$ and $d_i < d_j$ if $i <j$.
Then endow
any edge~$e$ joining the vertices $v_1$ and $v_2$ with
$\ell(v_1) > \ell(v_2)$ with length
$$|e| \= (d_{\ell(v_1)} - d_{\ell(v_2)})/(k(e^+)+1)\,.$$
\end{enumerate}
This algorithm is effective but not efficient, since most of
the enhanced level graphs that are built in the process will
be discarded in the end.

%% file: pic_graph_3.tex
\begin{figure}[htb]
	\centering
	\begin{minipage}[t]{0.4\textwidth}
		\centering
		\begin{tikzpicture} [scale=2]
		\tikzstyle{every node}=[font=\normalsize]
		
		\draw[rotate around={20:(.64,-.6)}] (.3,-.5) ellipse (.35cm and .05cm); 
		\begin{scope}[xshift=1.7cm]
		\draw[rotate around={20:(.64,-.6)}] (.3,-.5) ellipse (.35cm and .05cm); 
		\end{scope}
		\begin{scope}[xshift=1.7cm,yshift=1cm]
		\draw[rotate around={20:(.64,-.6)}] (.3,-.5) ellipse (.35cm and .05cm); 
		\end{scope}
		\begin{scope}[yshift=1.2cm]
		\draw[rotate around={20:(.64,-.6)}] (.3,-.5) ellipse (.35cm and .05cm); 
		\end{scope}
		
		\path[draw,thick] (.599,-.5) node(A)[inner sep=0]{} [dotted] -- ++(0,1.2) node(B)[inner sep=0]{};
		\path[draw] (.599,.7) -- ++(-45:1.356) node(C)[inner sep=0]{};
		\path[draw] (C) [inner sep=0]{}-- ++(45:1.05) node(F)[inner sep=0]{}; 
		
		\draw [dotted,thick] (-.04,-.75) node(H)[inner sep=0]{} -- (-.04,.46) node(J)[inner sep=0]{};
		
		\draw (-.04,.46) -- ++(-45:.8) node(K)[inner sep=0]{} 
		-- ++(0:.78) node(L)[inner sep=0]{}
		-- ++(45:.5) node(D)[inner sep=0]{};		
		
		\draw [dotted,thick] (1.66,.25) -- (1.66,-.75) node(I)[inner sep=0]{};
		
		\draw [dotted,thick] (2.3,.5) node(F)[inner sep=0]{} -- (2.3,-.5) node(G)[inner sep=0]{}; 		
		
		\path[draw, name path=lo] (.58,-.495) -- (2.3,-.495); 
		\path[draw, name path=lu] (-.04,-.75) -- (1.67,-.75); 
		
		\draw[] (C) -- ++(225:.2)
		(C) -- ++(-45:.2)
		(K) -- ++(225:.2)
		(L) -- ++(-45:.2); 
		
		\path [draw,dotted,thick, name path=p3] (C) --++ (-90:.25) node(N)[inner sep=0]{};
		\fill[name intersections={of=lo and p3}] (intersection-1) circle (1pt)
		node[below,yshift=-2pt]{2};
		\path [draw,dotted,thick, name path=p1] (K) --++ (-90:.86) node(O)[inner sep=0]{};
		\fill[name intersections={of=lu and p1}] (intersection-1) circle (1pt);                	
		\path [draw,dotted,thick, name path=p2] (L) --++ (-90:.86) node(P)[inner sep=0]{};	
		\fill[name intersections={of=lu and p2}] (intersection-1) circle (1pt);    					
		
		\fill 
		(B) circle (1pt) (C) circle (1pt) 
		(D) circle (1pt)  (F) circle (1pt) 
		(J) circle (1pt) (K) circle (1pt) (L) circle (1pt);
		\draw (-.3,.65) node {$\Gamma^+$};
		\draw (-.3,-.85) node {$\Gamma^{\phantom{+}}$};
		
		\end{tikzpicture}
	\end{minipage}
	\caption{Illustration of the edge condition}
	\label{cap:IllEdge}
\end{figure}
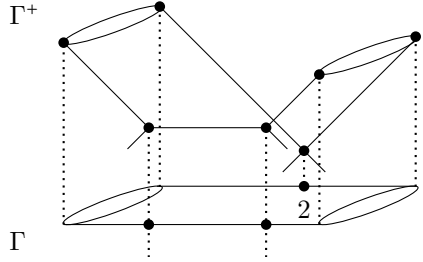

%% file: pic_graph_1.tex
\begin{figure}[htb]
	\centering
	\begin{minipage}[t]{.38\linewidth}
		\centering
		\begin{tikzpicture}[scale=2]
        \tikzstyle{every node}=[font=\normalsize]
        \draw (1.3,-.7) node[] {\phantom{}}; 
        \draw (2.3,1.3) node[] {\phantom{}}; 
        
        \draw (.3,-.5) ellipse (.3cm and .05cm); 
        \draw (2,-.5) ellipse (.3cm and .05cm); 
        \draw (2,1.26) ellipse (.3cm and .05cm); 
        \draw ($(.3,-.2) + (70:.3cm and .05cm)$(P) arc (70:250:.3cm and .05cm);
		\path[draw,thick] (.599,-.5) node(A)[inner sep=0]{} [dotted]
		-- ++(0,.66) node(B)[inner sep=0]{};
		\path[draw,thick] (.599,-.5) node(A)[inner sep=0]{} [dotted]
		-- ++(0,.66) node(B)[inner sep=0]{};
        
		\path[draw,thick] (1.7,1.26) node(D)[inner sep=0]{} [dotted]
		-- ++(0,-1.76) node(E)[inner sep=0]{}; 
        \path[draw] (E) [] -- ++(-1.1,0);
        %
		\path[draw,dotted,thick] (.2,-.54) node(P1)[inner sep=0]{}
		-- ++ (.2,.08) node(P2)[inner sep=0]{} 
        -- ++ (0,.3) node(P3)[inner sep=0]{} -- ++ (-.2,-.08) node(P4)[inner sep=0]{} -- ++ (0,-.3);
		\fill 
		(B) circle (1pt)
		(D) circle (1pt);
        \fill (P1) circle (1pt) (P2) circle (1pt) (P3) circle (1pt) (P4) circle (1pt);
        \path[draw] (P3) [] -- ++(-45:.15);
        \path[draw] (P4) [] -- ++(-45:.15);
        \path[draw] (P4) [] -- ++(45:2.1);
        \draw (P3) -- (B);
   		\draw (-.3,.65) node {$\Gamma^+$};
   		\draw (-.3,-.5) node {$\Gamma^{\phantom{+}}$};
		\end{tikzpicture}
	\end{minipage}%
	\hspace{15mm}
	\begin{minipage}[t]{.38\linewidth}
		\centering
		\begin{tikzpicture}[scale=2]
        \tikzstyle{every node}=[font=\normalsize]
        \draw (1.3,-.7) node[] {\phantom{}}; 
        \draw (2.3,1.3) node[] {\phantom{}}; 

        \draw (.3,-.5) ellipse (.3cm and .05cm); 
        \draw (2,-.5) ellipse (.3cm and .05cm); 
        \draw (2,.5) ellipse (.3cm and .05cm); 
        \draw (.3,.7) ellipse (.3cm and .05cm); 
        \path[draw,thick] (.599,-.5) node(A)[inner sep=0]{} [dotted] -- ++(0,1.2) node(B)[inner sep=0]{};
        
        \path[draw] (B) -- ++(-45:.92) node(C)[inner sep=0]{};
        \path[draw] (C) [dotted] -- ++(-90:.56) node(H) [inner sep=0]{};
		\path[draw] (C) -- ++(45:.64) node(D)[inner sep=0]{}; 
        \path[draw,thick] (D) [dotted] -- ++(0,-1) node(E)[inner sep=0]{}; 
        \path[draw] (E) [] -- ++(-1.1,0);
        
        \path[draw] (C) [] -- ++(-45:.15);        
        \path[draw] (C) [] -- ++(225:.15);      
      
		\fill 
		(B) circle (1pt) (C) circle (1pt) 
        (D) circle (1pt) (H) circle (1pt) node[below] {$2$};
   		\draw (-.3,.65) node {$\Gamma^+$};
   		\draw (-.3,-.5) node {$\Gamma^{\phantom{+}}$};
		\end{tikzpicture}
	\end{minipage}
\caption{Realizability locus over the dumbbell graph.}
\label{cap:DumbOK}
\end{figure}
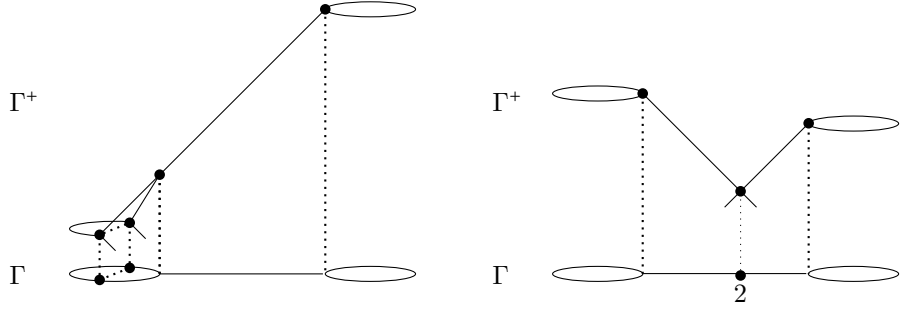

%% file: pic_graph_2.tex
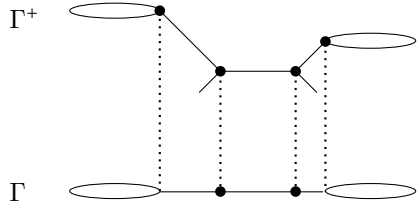
\begin{figure}[htb]
	\centering
	\begin{minipage}[t]{.40\linewidth}
		\centering
		\begin{tikzpicture}[scale=2]
        \tikzstyle{every node}=[font=\normalsize]
        \draw (.3,-.5) ellipse (.3cm and .05cm); 
        \draw (2,-.5) ellipse (.3cm and .05cm); 
        \draw (2,.5) ellipse (.3cm and .05cm); 
        \draw (.3,.7) ellipse (.3cm and .05cm); 
        \path[draw,thick] (.599,-.5) node(A)[inner sep=0]{} [dotted] -- ++(0,1.2) node(B)[inner sep=0]{};
        \path[draw] (B) -- ++(-45:.57) node(C)[inner sep=0]{};
        \path[draw] (C) -- ++(0:.5) node(D)[inner sep=0]{}; 
        \path[draw] (D) -- ++(45:.28) node(E)[inner sep=0]{}; 
        \path[draw,thick] (E) [dotted] -- ++(0,-1) node(F)[inner sep=0]{};
        \path[draw] (F) [] -- ++(-1.1,0);
        \path[draw] (C) [] -- ++(225:.2);
        \path[draw] (D) [] -- ++(-45:.2);        
        \path[draw,thick] (C) [dotted] -- ++(90:-.8) node(G)[inner sep=0]{};
        \path[draw,thick] (D) [dotted] -- ++(90:-.8) node(H)[inner sep=0]{};
		\fill 
		(B) circle (1pt) (C) circle (1pt) (D) circle (1pt) 
        (E) circle (1pt) (G) circle (1pt) (H) circle (1pt);;
   		\draw (-.3,.65) node {$\Gamma^+$};
   		\draw (-.3,-.5) node {$\Gamma^{\phantom{+}}$};
		\end{tikzpicture}
	\end{minipage}%
	\caption{A non-realizable configuration over the dumbbell graph.}
	\label{cap:DumbBAD}
\end{figure}

%% file: pic_graph_4.tex
\begin{figure}[htb]
	\centering
	\begin{minipage}[t]{0.38\textwidth}
		\centering
		\begin{tikzpicture} [scale=2]
		\tikzstyle{every node}=[font=\normalsize]
	    \draw (1,-1.1) node[] {\phantom{}}; 
		\path[draw] (0,-.9) node(A)[inner sep=0]{} 
		--++(70:.33) node(B)[inner sep=0]{}
		--++(-8:.9) node(C)[inner sep=0]{}
		--++(A);
		\path[draw] (0,.5) node(A')[inner sep=0]{} 
		--++(70:.33) node(B')[inner sep=0]{}
		--++(-8:.9) node(C')[inner sep=0]{}
		--++(A');
		
		\draw (A) -- ++(5.8:1.88) node(D)[inner sep=0]{} -- ++(B);
		\draw[] (C) -- (D) node [midway,xshift=-10pt,yshift=.-7pt]{\rotatebox{-50}{$\longleftarrow$}};
		\draw (C) -- (D) node [midway,below,xshift=0cm,yshift=-.3cm]{$e_6$};

		\draw (A') -- ++(-45:1) node(E)[inner sep=0]{}  -- ++(0:.8) node(F)[inner sep=0]{};
		\draw (B') -- ++(-45:1) node(G)[inner sep=0]{}   -- ++(0:.8) node(H)[inner sep=0]{};
		\draw (H) --++(45:.35) node(J)[inner sep=0]{};
		\draw (F) -- (J);
		\draw (C') -- (J);
		
		\draw[dotted,thick] (A) -- (A') (B) -- (B')	(C) -- (C'); 
		\draw[dotted,thick] (D) -- (J);
		\draw[dotted,thick] (E) --++(-90:.63) node(E')[inner sep=0]{};
		\draw[dotted,thick] (F) --++(-90:.545) node(F')[inner sep=0]{};
		\draw[dotted,thick] (G) --++(-90:.742) node(G')[inner sep=0]{};
		\draw[dotted,thick] (H) --++(-90:.79) node(H')[inner sep=0]{};
		
		\fill 
		(A') circle (1pt) (B') circle (1pt) (C') circle (1pt) 
		(E) circle (1pt) (F) circle (1pt) 
		(G) circle (1pt) (H) circle (1pt) 
		(J) circle (1pt)
		(E') circle (1pt) (F') circle (1pt) (G') circle (1pt) (H') circle (1pt);
		
		\path[draw] (E) -- ++(225:.2);
		\path[draw] (F) -- ++(-45:.2);
		\path[draw] (G) -- ++(225:.2);
		\path[draw] (H) -- ++(-45:.2);
		
		\draw (A) -- (B) node [midway,left]{$e_1$}
		(B) -- (C) node [midway,below,xshift=-.2cm,yshift=.1cm]{$e_2$}
		(C) -- (A) node [midway,above,xshift=-.5cm,yshift=-.15cm]{$e_3$}
		(B) -- (D) node [midway,above,xshift=-4pt]{$e_4$}
		(A) -- (D) node [midway,below]{$e_5$};
		
		\draw (-.3,.65) node {$\Gamma^+$};
		\draw (-.3,-.85) node {$\Gamma^{\phantom{+}}$};
		
		\end{tikzpicture}
	\end{minipage}
	\hspace{6mm}  
	\begin{minipage}[t]{0.34\textwidth}
		\centering
		\begin{tikzpicture} [scale=2]
		\tikzstyle{every node}=[font=\normalsize]
	    \draw (1,-1.1) node[] {\phantom{}}; 
		\path[draw] (0,-.9) node(A)[inner sep=0]{} 
		--++(70:.33) node(B)[inner sep=0]{}
		--++(-20:.4) node(C)[inner sep=0]{}
		--++(A);
		\path[draw] (0,.5) node(A')[inner sep=0]{} 
		--++(70:.33) node(B')[inner sep=0]{}
		--++(-20:.4) node(C')[inner sep=0]{}
		--++(A');
		
		\coordinate (D) at ($(A)+(5.8:1.88)$);		
		\draw (A) -- (D) node (lineA_D)[]{};
		\draw (D) -- (B) node (lineB_D)[]{} ;
		\draw (C) -- (D) node (lineC_D)[]{};
					
		\draw (B') -- ++(-45:.7) node(G)[inner sep=0]{} --++(0:1.25) node(J)[inner sep=0]{};
		\draw[] (C') --++(-45:.8) node(E_neu)[inner sep=0]{};
		\draw[] (A') --++(-45:.55) node(K)[inner sep=0]{} -- (J);
		\draw[] (J) --++(-135:.29) node(F_neu)[inner sep=0]{};
		\draw[] (E_neu) -- (F_neu);

		\draw[dotted,thick] (A) -- (A') (B) -- (B')	(C) -- (C'); 
		\draw[dotted,thick] (D) -- (J);
		\draw[dotted,thick] (G) --++(-90:.94) node(G')[inner sep=0]{};
		\draw[dotted,thick] (E_neu)-- (E_neu  |- lineC_D) node (E_neu')[]{} ;
		\draw[dotted,thick] (F_neu)-- (F_neu  |- lineC_D) node (F_neu')[]{} ;
	    \draw[dotted,thick] (K) --++(-90:.97) node(K')[inner sep=0]{};
		
		\fill 
		(A') circle (1pt) (B') circle (1pt) (C') circle (1pt) 
		(E_neu) circle (1pt) (F_neu) circle (1pt) 
		(G) circle (1pt) 
		(J) circle (1pt)
		(G') circle (1pt)
		(E_neu') circle (1pt)
		(F_neu') circle (1pt) (K) circle (1pt) (K') circle (1pt); 
		
		\path[draw] (E_neu) -- ++(225:.2);
		\path[draw] (F_neu) -- ++(-45:.2);
		\path[draw] (G) -- ++(225:.2);
		\path[draw] (K) -- ++(225:.2);

		\draw (-.3,.65) node {$\Gamma^+$};
		\draw (-.3,-.85) node {$\Gamma^{\phantom{+}}$};
		
		\end{tikzpicture}
	\end{minipage}

\caption{Realizable configurations of maximal dimension on~$K_4$:
  $3$-cycles on top level, edges with two points}
	\label{cap:K4realize1}
\end{figure}
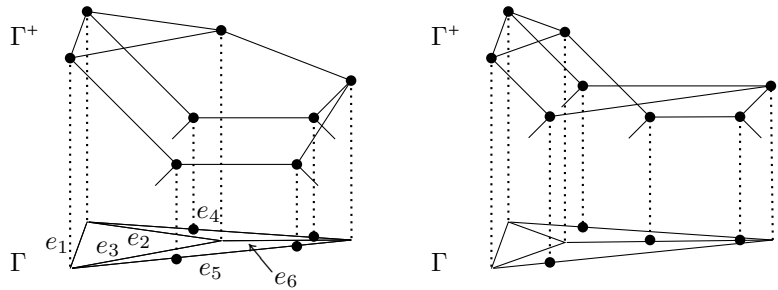

%% file: pic_graph_6.tex

\begin{figure}[htb]
	\centering
	\begin{minipage}[t]{0.38\textwidth}
		\centering
		\begin{tikzpicture} [scale=2]
		\tikzstyle{every node}=[font=\normalsize]
		\draw (1,-1.1) node[] {\phantom{}}; 
		\path[draw] (0,-.9) node(A)[inner sep=0]{} 
		--++(70:.33) node(B)[inner sep=0]{}
		--++(-20:.4) node(C)[inner sep=0]{}
		--++(A);
		\path[draw] (0,.5) node(A')[inner sep=0]{} 
		--++(70:.33) node(B')[inner sep=0]{}
		--++(-20:.4) node(C')[inner sep=0]{}
		--++(A');
		
		\coordinate (D) at ($(A)+(5.8:1.88)$);		
		\draw (A) -- (D) node (lineA_D)[]{};
		\draw (D) -- (B) node (lineB_D)[]{} ;
		\draw (C) -- (D) node (lineC_D)[]{};
		
		\draw (B')  ++(-45:.6) node(G)[inner sep=0]{} ++(0:1.34) node(J)[inner sep=0]{};
		\draw[] (C') --++(-45:.4) node(E_neu)[inner sep=0]{};
		\draw[] (A') --++(-45:.55) node(K)[inner sep=0]{};
		\draw[blue] (K) --++(-45:.4) node(L)[inner sep=0]{};	
		\draw[] (A') -- (L);
		\draw[] (L) --++(0:.8) node(M)[inner sep=0]{};
		\draw[] (M) --++(45:.4) node(N)[inner sep=0]{};
		\draw[] (B') -- (J); 
		\draw (E_neu) -- (J);	
		\draw[] (J) -- (N); 			
		
		\draw[dotted,thick] (A) -- (A') (B) -- (B')	(C) -- (C'); 
		\draw[dotted,thick] (D) -- (J);
		\draw[dotted,thick] (E_neu)-- (E_neu  |- lineC_D) node (E_neu')[]{} ;
		\draw[dotted,thick] (N) --++(-90:.825) node(N')[inner sep=0]{};
		\draw[dotted,thick] (L)--++(-90:.66) node(L')[inner sep=0]{};		
		\draw[dotted,thick] (M)--++(-90:.58) node(M')[inner sep=0]{};		
		
		\fill
		(A') circle (1pt) (B') circle (1pt) (C') circle (1pt) 
		(E_neu) circle (1pt) 
		(J) circle (1pt)
		(E_neu') circle (1pt)
		(L) circle (1pt) (L') circle (1pt) (M) circle (1pt) (M') circle (1pt)
		(N) circle (1pt) (N') circle (1pt); 
		
		\path[draw] (E_neu) -- ++(225:.2);
		\path[draw] (M) -- ++(-45:.2);
		\path[draw] (L) -- ++(225:.2);
		\path[draw] (N) -- ++(-45:.2);
		
		\draw (-.3,.65) node {$\Gamma^+$};
		\draw (-.3,-.85) node {$\Gamma^{\phantom{+}}$};
		
		\end{tikzpicture}
	\end{minipage}
        \hspace{6mm}
        	\begin{minipage}[t]{0.38\textwidth}
		\centering
		\begin{tikzpicture} [scale=2]
		\tikzstyle{every node}=[font=\normalsize]
		\draw (1,-1.1) node[] {\phantom{}}; 
		\path[draw] (0,-.9) node(A)[inner sep=0]{} 
		--++(8:1.4) node(B)[inner sep=0]{}
		--++(-130:.4) node(C)[inner sep=0]{}
		--++(A);
		\path[draw] (B)	--++(-29:.4) node(D)[inner sep=0]{}
		--++(C);
		\path[draw] (0,.5) node(A')[inner sep=0]{} 
		--++(8:1.4) node(B')[inner sep=0]{}
		--++(-130:.4) node(C')[inner sep=0]{}
		--++(A');
		\path[draw] (B')	--++(-50:.55) node(D')[inner sep=0]{}
		--++(C');

		\draw (A) -- (D) node (lineA_D)[]{};
		\path[draw] (A') --++(-45:1.03) node(E)[inner sep=0]{} --++ (0:.76);
		\path[] (D') --++(-115:.55) node(F)[inner sep=0]{};
		
		\path[draw] (F) --++(45:.12) node(H)[inner sep=0]{}
		(H) --++(65:.2) node(G)[inner sep=0]{};
		\draw[] (G) -- (D');

		\fill (A') circle (1pt) (B') circle (1pt) (C') circle (1pt) 
		(D') circle (1pt) (E) circle (1pt) (F) circle (1pt)
		(G) circle (1pt) (H) circle (1pt);
		
		\draw[dotted,thick] (A) -- (A') (B) -- (B')	(C) -- (C') (D) -- (D'); 
		\draw[dotted,thick] (E)-- (E  |- lineA_D) node (E')[]{} ;
		\draw[dotted,thick] (F)-- (F  |- lineA_D) node (F')[]{} ;
		\draw[dotted,thick] (G)-- (G  |- lineA_D) node (G')[]{} ;
		\draw[dotted,thick] (H)-- (H  |- lineA_D) node (H')[]{} ;		
		
		\fill[] (E') circle (1pt) (F') circle (1pt) (G') circle (1pt) (H') circle (1pt); 
		\fill[] (G') circle (1pt) (H') circle (1pt);
		
		\path[draw] (E) -- ++(225:.2);
		\path[draw] (F) -- ++(-45:.2);
		\path[draw] (G) -- ++(-45:.2);
		\path[draw] (H) -- ++(-45:.2);
		
		\draw (-.3,.65) node {$\Gamma^+$};
		\draw (-.3,-.85) node {$\Gamma^{\phantom{+}}$};
		
		\end{tikzpicture}
	\end{minipage}
	        \caption{Realizable configurations of maximal
                  dimension on~$K_4$:
                 $3$-cycles on top level, more than two points on some edge}
	\label{cap:K4realize2}
\end{figure}
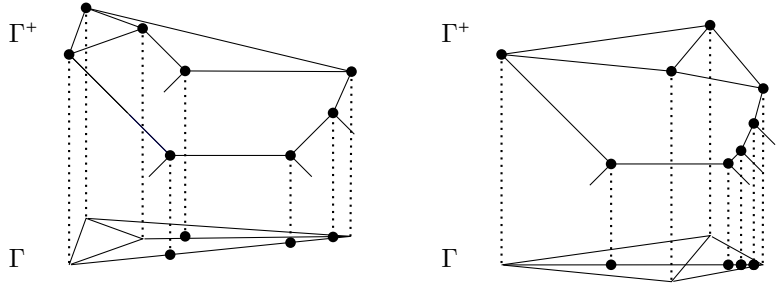

%% file: pic_graph_5.tex

\begin{figure}[htb]
  \centering
\begin{minipage}[t]{0.36\textwidth}
	\centering
	\begin{tikzpicture} [scale=2]
	\tikzstyle{every node}=[font=\normalsize]
	\draw (1,-1.1) node[] {\phantom{}}; 
	
	\draw(0,.4) node(A)[inner sep=0]{} -- ++(-45:.8) node(B)[inner sep=0]{}  
	-- ++(0:.8) node(C)[inner sep=0]{} -- ++(45:.8) node(D)[inner sep=0]{} -- (A);
	
	\draw[](.4,.4) node(E)[inner sep=0]{} -- ++(-45:.55) node(F)[inner sep=0]{}  
	-- ++(0:.9) node(G)[inner sep=0]{} -- ++(45:.55) node(H)[inner sep=0]{} -- (E);
	
	\draw ($(1.04,.27) + (40:1.37cm and .2cm)$(I) arc (40:140:1.37cm and .4cm);
	
	\draw[dotted,thick] (B) --++(-90:.55) node(B')[inner sep=0]{} 
	(C) --++(-90:.55) node(C')[inner sep=0]{};
	
	\draw[draw=none] (B') -- (C');
	\path[draw, name path=1m] (0,-.665) -- (2.085,-.665);
	
	\path[draw,dotted,thick,name path=p9] (E) --++(-90:1.06) node(E')[inner sep=0]{}; 
	\path[draw,dotted,thick,name path=p10] (D) --++(-90:1.06) node(D')[inner sep=0]{};

	\begin{scope}[yshift=-1.09cm]
	\draw ($(1.03,.3) + (40:1.37cm and .2cm)$(I) arc (40:140:1.36cm and .4cm);
	\end{scope}
	\path[draw,name path=p1] ($(.73,-.54) + (-40:1.57cm and .2cm)$(J) arc (-40:-140:1.26cm and .4cm);
	
	\path[draw,name path=p6] ($(1.24,-.54) + (-40:1.1cm and .2cm)$(J) 
	arc (-40:-140:1.1cm and .44cm);
	
	\path[draw,dotted,thick,name path=p2] (B') --++(-90:.1) node(B'')[inner sep=0]{}; 
	\fill[name intersections={of=p1 and p2}] (intersection-1) circle (1pt);
	\path[draw,dotted,thick,name path=p3] (C') --++(-90:.08) node(C'')[inner sep=0]{};
	\fill[name intersections={of=p1 and p3}] (intersection-1) circle (1pt);
	\path[draw,dotted,thick,name path=p4] (F) --++(-90:.83) node(F')[inner sep=0]{}; 
	\fill[name intersections={of=p4 and p6}] (intersection-1) circle (1pt);
	\path[draw,dotted,thick,name path=p5] (G) --++(-90:.84) node(G')[inner sep=0]{}; 
	\fill[name intersections={of=p5 and p6}] (intersection-1) circle (1pt);
	
	\path[draw,dotted,thick,name path=p7] (A) --++(-90:1.08) node(A')[inner sep=0]{}; 
	\path[draw,dotted,thick,name path=p8] (H) --++(-90:1.08) node(H')[inner sep=0]{}; 
	
	\fill[] (B) circle (1pt) (C) circle (1pt)
	(F) circle (1pt) (G) circle (1pt);        
	\fill[] (A) circle (1pt) (D) circle (1pt)
	(E) circle (1pt) (H) circle (1pt);      
	
	\path[draw] (B) -- ++(225:.18);
	\path[draw] (C) -- ++(-45:.18);
	\path[draw] (F) -- ++(225:.18);
	\path[draw] (G) -- ++(-45:.18);	
	
	\draw (-.3,.65) node {$\Gamma^+$};
	\draw (-.3,-.85) node {$\Gamma^{\phantom{+}}$};
	
	\end{tikzpicture}	
\end{minipage}
\caption{Realizable configurations of maximal dimension on~$K_4$:
$4$-cycle on top level}
	\label{cap:K4realize3}
\end{figure}
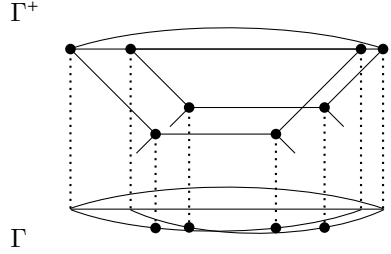

%% file: pic_dim_dumbbell.tex
\begin{figure}[htb]
	\centering
	\begin{minipage}[t]{0.5\textwidth}
		\centering
		\begin{tikzpicture} [scale=.6]
		\tikzstyle{every node}=[font=\normalsize]
		\tikzset{arrow/.style={latex-latex}}
		\def\S{2.825cm} 
		
		
		\path coordinate (A) at (0,2) (A) -- +(45:\S) 
		coordinate (B) (B) -- +(-45:\S) 
		coordinate (C) (C) -- +(-135:\S) 
		coordinate (D);
		\coordinate (E) at (-2,2); 
		\coordinate (F) at (6,2);
		
		\draw (A) -- (B) -- (C) -- (D) -- (A);	
		\foreach \x in {A,B,C,D}
		\fill (\x) circle (3pt);
		\draw[line width=1.3pt] (E) -- (F); 
		\draw (B) edge [arrow,thick,shorten >=12pt,shorten <=12pt,bend left =25] (D);
		
		\end{tikzpicture}
	\end{minipage}
	\caption{The simplices over the dumbbell graph}
	\label{cap:dimdumb}
\end{figure}
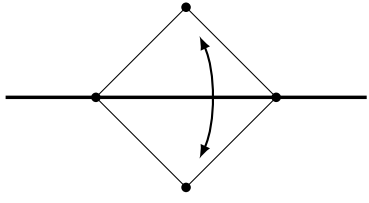

%% file: pic_deg_1.tex
\begin{figure}[htb]
\begin{minipage}[t]{0.4\textwidth}
\centering
\begin{tikzpicture} [scale=2]
\tikzstyle{every node}=[font=\normalsize]

\coordinate[label=left:$g$] (A) at (.5,0);
\coordinate (dots_lo) at (.7,.25); 
\coordinate (dots_lu) at (.7,-.3); 
\fill (A) circle (1pt);
\path[draw] (A) --++(30:.6) node[above](bp){$b_p$}
(A) --++(-30:.6) node[below](ar){$a_r$}
(A) --++(90:.3) node[above](b2){$b_2$}
(A) --++(-90:.3) node[below](a2){$a_2$}
(A) --++(120:.35) node[above](b1){$b_1$}
(A) --++(-120:.35) node[below](a1){$a_1$};
\draw (dots_lo) node []{$\cdots$}
(dots_lu) node []{$\cdots$};

\draw[] (1.3,0) node [scale=2]{$\rightrsquigarrow$};

\begin{scope}[shift={(1.5,0)}]
\coordinate[label=left:$g-1$] (A) at (.5,.15);
\coordinate[label=left:{$g=0$}] (B) at (.5,-.15);
\coordinate (dots_ro) at (.7,.4); 
\coordinate (dots_ru) at (.7,-.45); 
\draw (.5,0) ellipse (.06cm and .15cm);
\fill (A) circle (1pt);
\fill (B) circle (1pt);
\path[draw] (A) --++(30:.6) node[above](bp){$b_p$}
(B) --++(-30:.6) node[below](ar){$a_r$}
(A) --++(90:.3) node[above](b2){$b_2$}
(B) --++(-90:.3) node[below](a2){$a_2$}
(A) --++(120:.35) node[above](b1){$b_1$}
(B) --++(-120:.35) node[below](a1){$a_1$};
\draw (dots_ro) node []{$\cdots$}
(dots_ru) node []{$\cdots$};
\end{scope}
\end{tikzpicture}
\end{minipage}
%
%
%
\caption{Splitting positive genus}
	\label{cap:deg1}
\end{figure}
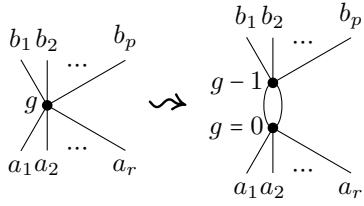

%% file: pic_deg_2.tex
\begin{figure}[h]
\hspace{.1mm} 
\begin{minipage}[t]{0.42\textwidth}
	\centering
	\begin{tikzpicture} [scale=1.5]
	\tikzstyle{every node}=[font=\normalsize]
	\coordinate[] (A) at (.5,0);
	\coordinate (dots_lo) at (.65,.25);
	\fill (A) circle (1pt);
	\path[draw] (A) --++(45:.42) node[above](br){$b_r$}
	(A) --++(90:.3) node[above](b2){$b_2$}
	(A) --++(120:.35) node[above](b1){$b_1$}
	(A) --++(-135:.35) node[below](a){$a$}
	(A) --++(-45:.42) node[inner sep=0](B){}
	(B) --++(-135:.35) node[below](0){$0$}
	(B) --++(-45:.35) node[inner sep=0](C){}
	(C) --++(-135:.35) node[below](D){$0$}
	(C) --++(-45:.35) node[below](E){$0$};
	\draw (dots_lo) node []{$\cdots$};
	\fill (B) circle (1pt);
	\fill (C) circle (1pt);
	\draw[] (1.45,0) node [scale=2]{$\rightrsquigarrow$};
	\coordinate[] (F) at (2,0);
	\coordinate (dots_lo) at (2.12,.25); 
	\fill (F) circle (1pt);
	\path[draw] (F) --++(45:.42) node[above](br){$b_r$}
	(F) --++(90:.3) node[above](b2){$b_2$}
	(F) --++(-90:.25) node[below](a2){$0$}
	(F) --++(120:.35) node[above](b1){$b_1$}
	(F) --++(-120:.29) node[below,yshift=-2pt](a){$a$}
	(F) --++(-60:.29) node[below](a){$0$}
	(F) --++(-30:.5) node[below](a){$0$};
	\draw (dots_lo) node []{$\cdots$};
	\draw[] (2.7,0) node [scale=2]{$\leftlsquigarrow$};
	\coordinate[] (G) at (3.2,0);
	\coordinate (dots_lo) at (3.31,.25);
	\fill (G) circle (1pt);
	\path[draw] (G) --++(45:.42) node[above](br){$b_r$}
	(G) --++(90:.3) node[above](b2){$b_2$}
	(G) --++(120:.35) node[above](b1){$b_1$}
	(G) --++(-135:.35) node[below](0){$0$}
	(G) --++(-45:.42) node[inner sep=0](H){}
	(H) --++(-135:.35) node[below](0){$0$}
	(H) --++(-45:.35) node[inner sep=0](I){}
	(I) --++(-135:.35) node[below](J){$0$}
	(I) --++(-45:.35) node[below,yshift=-2pt](K){$a$};
	\draw (dots_lo) node []{$\cdots$};
	\fill (H) circle (1pt);
	\fill (I) circle (1pt);
	\end{tikzpicture}
\end{minipage}

\caption{Rearranging trees of marked points}
	\label{cap:deg2}
\end{figure}
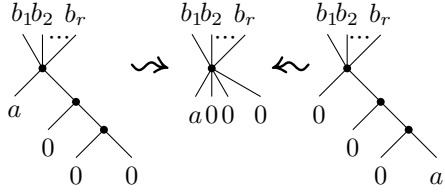

%% file: pic_4_dumbbell.tex
\begin{figure}[htb]
	\centering
	\begin{minipage}[t]{1\linewidth}
		\centering
		\begin{tikzpicture}[scale=2]
		\tikzstyle{every node}=[font=\normalsize]
		\fill[] (0,0) circle (1pt);	
		\draw (.3,0) ellipse (.3cm and .05cm); 
		\draw (1.65,0) ellipse (.3cm and .05cm); 
		\path[draw] (.6,0) [] -- ++(.75,0);
	    \fill[] (3.2,0) circle (1pt);	
	    \draw (2.9,0) ellipse (.3cm and .05cm); 
	    \draw (4.25,0) ellipse (.3cm and .05cm); 
	    \path[draw] (3.2,0) [] -- ++(.75,0);
		\fill[] (2.25,-.2) circle (1pt);	
		\draw (1.2,-.2) ellipse (.3cm and .05cm); 
		\draw (2.55,-.2) ellipse (.3cm and .05cm); 
		\path[draw] (1.5,-.2) [] -- ++(.75,0);    
	    \fill[] (5.45,-.2) circle (1pt);	
	    \draw (3.8,-.2) ellipse (.3cm and .05cm); 
	    \draw (5.15,-.2) ellipse (.3cm and .05cm); 
	    \path[draw] (4.1,-.2) [] -- ++(.75,0);
		\end{tikzpicture}
	\end{minipage}%
	\caption{The realizability locus $\PP \cRR_\Omega(2)$}
	\label{cap:realOm2}
\end{figure}
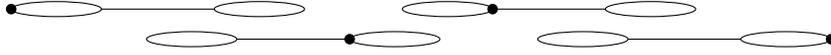